\documentclass{article}

\usepackage[utf8]{inputenc}
\usepackage{amsmath, amssymb, amsthm, bm, mathrsfs}
\usepackage[top=1.5cm, bottom=1.5cm, left=2cm, right=2cm]{geometry}
\usepackage{csquotes, authblk, cases, enumitem}
\usepackage{graphicx, float, caption, subcaption, adjustbox}
\usepackage[hidelinks]{hyperref}
\usepackage{cleveref}
\usepackage[numbers,sort]{natbib}
\usepackage[section, plain]{algorithm}
\usepackage{algpseudocode}
\usepackage{array, tabularx, colortbl}
\usepackage{color, xcolor}
\usepackage[hang,flushmargin]{footmisc}
\usepackage[normalem]{ulem}
\counterwithin{figure}{section}
\counterwithin{table}{section}
\numberwithin{equation}{section}
\graphicspath{{Figures/}}
\definecolor{darkred}{rgb}{0.8,0,0}
\DeclareMathOperator{\Span}{\operatorname{span}}

\DeclareMathOperator{\diag}{\operatorname{diag}}

\newenvironment{frcseries}{\fontfamily{frc}\selectfont}{}

\theoremstyle{definition}
\newtheorem{definition}{Definition}[section]
\newtheorem{remark}[definition]{Remark}

\newtheorem{theorem}[definition]{Theorem}
\newtheorem{lemma}[definition]{Lemma}
\newtheorem{corollary}[definition]{Corollary}

\title{A theoretical analysis of mass scaling techniques}
\author[1]{Yannis Voet \thanks{yannis.voet@epfl.ch}}
\author[1]{Espen Sande \thanks{espen.sande@epfl.ch}}
\author[1]{Annalisa Buffa \thanks{annalisa.buffa@epfl.ch}}
\affil[1]{\small MNS, Institute of Mathematics, École polytechnique fédérale de Lausanne, Station 8, CH-1015 Lausanne, Switzerland}
\date{\today}

\begin{document}

\maketitle

\begin{abstract}
Mass scaling is widely used in finite element models of structural dynamics for increasing the critical time step of explicit time integration methods. While the field has been flourishing over the years, it still lacks a strong theoretical basis and mostly relies on numerical experiments as the only means of assessment. This contribution thoroughly reviews existing methods and connects them to established linear algebra results to derive rigorous eigenvalue bounds and condition number estimates. Our results cover some of the most successful mass scaling techniques, unraveling for the first time well-known numerical observations.

\end{abstract}

\section{Introduction and background}
Finite element analysis (FEA) has long established itself as an indispensable tool in many engineering disciples. Particularly in structural dynamics, one of its founding applications, where the method is extensively used for simulating the deformation and vibration of various structures and solids, including plates, shells and beams. In this context, explicit time integration is favored in several circumstances. Indeed, the dynamics of the problem often impose a natural restriction on the step size \cite{olovsson2005selective}, and enlarging it may lead to convergence issues within implicit unconditionally stable methods, specifically for nonlinear problems requiring the solution of a nonlinear system of equations at each iteration. On the contrary, the smaller step size used within explicit methods ensures greater robustness. Moreover, explicit methods offer the possibility of ``lumping'' the mass matrix enabling memory-efficient matrix-free implementations. \emph{Mass lumping} consists in directly substituting the mass matrix with an ad hoc diagonal approximation \citep{anitescu2019isogeometric,nguyen2023towards,voet2023mathematical,li2022significance}. Thus, explicit methods may avoid solving costly linear or nonlinear systems of equations, left alone the burden of linearizing the stiffness.

Unfortunately though, explicit methods are only conditionality stable and the higher frequencies impose a strong restriction on the critical time step. For example, for undamped dynamical systems, the critical time step of the central difference method is
\begin{equation}
\label{eq: CFL_central_difference}
    \Delta t_c= \frac{2}{\omega_n}
\end{equation}
where $\omega_n$ is the largest frequency of the discrete system \citep{hughes2012finite, bathe2006finite}. The largest discrete frequencies of classical $C^0$ finite element methods are wholly inaccurate and fortunately rarely contribute significantly to the solution. They form the so-called ``optical branches'' of the spectrum and diverge with the mesh size and polynomial degree \citep{gallistl2017stability,gahalaut2012condition}. Smooth isogeometric analysis (IGA) \cite{hughes2005isogeometric} features far fewer inaccurate frequencies \citep{cottrell2006isogeometric,cottrell2007studies,hughes2014finite,manni2022application}. They usually form noticeable spikes in the upper part of the spectrum and were coined \emph{outliers} for this very reason \cite{cottrell2006isogeometric}. Regardless of the discretization technique and smoothness, large inaccurate eigenvalues severely restrict the critical time step and removing or dampening them is paramount in explicit dynamics. The method for removing them, however, usually depends on the discretization technique. In isogeometric analysis, so-called \emph{outlier removal techniques} commonly exploit the smoothness and tensor product nature of spline spaces (see e.g. \citep{cottrell2006isogeometric,manni2022application,manni2023outlier,hiemstra2021removal,deng2021boundary,nguyen2022variational,voet2024robust}). These methods may also preserve the consistency of the discrete formulation and therefore higher rates of convergence. In more general cases with significantly less structure, one may resort to \emph{mass scaling}. Broadly speaking, mass scaling consists in dampening higher frequencies by altering the discrete formulation and, more specifically, the mass matrix. \emph{Conventional mass scaling} (CMS) \cite{key1971transient,hughes1978reduced}, which only scales the diagonal entries of the mass matrix, is the most straightforward way of achieving this goal. Unfortunately, it also heavily deteriorates the accuracy of the lower frequencies and mode shapes, and, subsequently, the numerical solution. \emph{Selective mass scaling} (SMS) is specifically designed for scaling down the highest frequencies while somewhat preserving the lowest ones. Thus, it still affects \emph{all} eigenvalues, but not uniformly. Mathematically speaking, the idea is reminiscent of inexact deflation techniques, which consist in removing unwanted eigenvalues from the spectrum of a matrix. Exact deflation techniques have been known for decades and originated from the early work of Hotelling \cite{hotelling1943some} for standard eigenproblems but are mostly impractical. Practical strategies, either global or local, are instead essentially heuristic and oftentimes little is known of their theoretical properties. Global strategies directly modify the assembled system matrices while local ones operate on the element matrices prior to assembly. The former are generally easier to analyze than the latter and are among the oldest methods. In the late 1990s, Macek and Aubert \cite{macek1995mass} and later Olovsson et al. \cite{olovsson2005selective} independently proposed the same global mass scaling procedure, based on linear fractional transformations of matrix pairs. The method was later generalized in \cite{tkachuk2014local} to rational polynomials, where the authors also stated optimality requirements for SMS. These methods are fully understood and transform (nonlinearly) the eigenvalues of a matrix pair while preserving its eigenvectors. Although the methods affect all eigenvalues, they may heavily damp the largest ones, while nearly preserving the smallest ones. In the same article, the authors also proposed an exact deflation method only scaling down the largest eigenvalues of the system. The idea was later revived and thoroughly improved in \cite{gonzalez2020large} but remains impractical, unless only a few well-separated eigenvalues must be deflated \cite{voet2024robust}.  As we will see, all of these methods are merely adaptions or reformulations of well-known linear algebra facts, mostly known since decades.

Since the early 2000s, focus has slowly shifted towards inexact global or local strategies. Global strategies often have a local counterpart, which simply consists in performing the same operations locally prior to assembly. Local strategies naturally suggest themselves when local properties dictate global ones. For instance, for hexahedral elements, the largest eigenvalues are tied to large aspect ratios \citep{olovsson2004selective,cocchetti2013selective,cocchetti2015selective} while for beam and plate models, they are associated to transverse shear modes \cite{oesterle2022intrinsically,oesterle2023finite,krauss2024intrinsically}. Already in the 1970s, several authors suggested locally scaling rotational components of lumped mass matrices within beam models \cite{key1971transient,hughes1978reduced}. This technique, nowadays known as \emph{rotational mass scaling} (RMS), merely applies conventional mass scaling to selected degrees of freedom and is rather inaccurate \cite{oesterle2023finite,krauss2024intrinsically}. As shown by Oesterle et al. \cite{oesterle2022intrinsically}, its accuracy is sometimes significantly improved by simply reparameterizing the model and choosing different primary unknowns. This concept of \emph{intrinsically selective mass scaling} (ISMS) was recently extended to plate formulations \cite{krauss2024intrinsically}.

For hexahedral elements, Olovsson et al. \cite{olovsson2004selective} first suggested grouping nodes aligned in the thickness direction of thin-walled structures and locally scaling the inverse mass matrices of each individual group. A rather similar strategy was followed by Cocchetti et al. for non-distorted \cite{cocchetti2013selective} and distorted \cite{cocchetti2015selective} hexahedral elements. Also in the early 2000s, Olovsson et al. \cite{olovsson2005selective} suggested instead locally scaling the element mass matrices of hexahedral finite elements and highlighted some appealing numerical properties. Nowadays, the method, together with a miscellaneous collection of variants and generalizations \citep{gavoille2013enrichissement,borrvall2011selective}, is strongly rooted in the engineering community and has been incorporated in commercial finite element software for car-crash simulations, such as LS-DYNA \cite{borrvall2011selective} and RADIOSS \cite{morancay2009dynamic}. Several improvements have also lately been suggested in \cite{hoffmann2023finite}. As we will see, some of these methods are (inexact) local deflation techniques. Exact local versions were investigated in \citep{tkachuk2014local,ye2017selective,gonzalez2018inverse,gonzalez2020large,voet2024robust}, sometimes by simply adapting their global counterpart. These methods are similar to eigenvalues stabilization techniques \cite{eisentrager2024eigenvalue}, which are one of numerous ad hoc stabilization schemes for immersed methods \citep{leidinger2020explicit,stoter2023critical}.

A completely different approach was followed in \cite{tkachuk2013variational} where the authors derived a parametrized family of mass matrices from a penalized Hamilton's principle. Interestingly, the method of Olovsson et al. may be recovered (up to a multiplicative factor) for specific parameter values and Ansatz spaces. However, the authors propose scaling the consistent mass matrix and the subsequent improvement of the critical time step may not even reach beyond the one for a lumped mass matrix. Additionally, the authors report severe conditioning issues, sometimes even preventing the convergence of iterative solvers. To our knowledge, this method has not been widely adopted in practice. In \cite{tkachuk2015direct,schaeuble2017variationally}, the authors used a similar strategy for building directly a sparse approximate inverse of the consistent mass (therein referred to as \emph{reciprocal mass matrix} (RMM)) and a variationally scaled version thereof. While this strategy alleviates the burden of solving linear systems with the scaled mass matrix, it also does not always increase the critical time step with respect to the lumped mass matrix. Moreover, the definiteness of the scaled approximate inverse practically limits the range of parameter values and prevents increasing the step size beyond 50\%. Finally, technical issues tied to dual polynomials further impede the method, such as the imposition of Dirichlet boundary conditions. A simpler alternative based on Lagrange multipliers was proposed for FEA \cite{gonzalez2018inverse} and later extended to IGA \cite{gonzalez2019inverse}.

Performing operations locally is often significantly faster than globally and offers potential for parallelization. However, while some global methods may be classified as ``exact'', most local methods rarely are due to the assembly process. Quantifying and analyzing their ``inexactness'' is one of the objectives of this article. Before that, the heterogeneous collection of methods listed above asks for a thorough review. Among the aforementioned mass scaling techniques, we have identified three common shortcomings: 
\begin{enumerate}
    \item Most practically relevant methods are inexact and will surely affect the smallest eigenfrequencies. The accuracy of the methods is often only verified by comparing them to the solution with a lumped mass matrix, which might not always yield an accurate approximation. Additionally, mass scaling may reduce the convergence rate of the smallest eigenvalues, if not already ruined by mass lumping. To our knowledge, few authors have undertaken a convergence study for their method.
    \item In addition to sometimes (critically) affecting the accuracy, mass scaling often leads to a non-diagonal scaled mass matrix, thereby partly undermining the computational savings from the reduction in the number of iterations. Surprisingly few authors have investigated this tradeoff and the time savings (if any) might come at the price of increased storage requirements. Some of the methods proposed are impractical and may not compete with a lumped mass, left alone accuracy concerns. Due to the non-diagonal scaled mass, researchers have resorted to preconditioned iterative methods \citep{olovsson2006iterative,tkachuk2014local} or Cholesky factorizations \citep{stoter2022variationally,stoter2023critical} for solving linear systems. The former has also driven condition number analyses for the scaled mass. 
    \item Mass scaling methods are first and foremost designed to increase the critical time step but theoretical bounds on the step size are rarely provided. Some heuristic and ``analytical'' step size estimates were derived for FEA \citep{cocchetti2013selective,cocchetti2015selective} and IGA \citep{hartmann2015mass,adam2015stable}, often by means of characteristic lengths. However, they generally do not guarantee a stable step size \cite{flanagan1984eigenvalues}, unless derived from rigorous bounds based e.g. on the Gershgorin circles \citep{flanagan1984eigenvalues,varga2011gersgorin,schaeuble2018time} or on Ostrowski’s bound \cite{tkachuk2019time}.
\end{enumerate}

In view of the points raised above, our article complements previous numerical studies with a theoretical analysis, which we feel is needed. Our analysis offers a unifying picture and a clearer positioning of the main methods proposed in the literature and is substantiated with some eigenvalue, step size and condition number estimates. Our study focuses exclusively on theoretical aspects and sheds light on the numerical observations reported in the literature. The few numerical experiments of this article are only meant to validate the theoretical predictions. Although mostly developed and applied for $C^0$ discretizations, some of the methods discussed herein are also valuable for smooth spline discretizations (i.e. IGA) and fit in the framework of outlier removal techniques provided one substitutes elements for patches.

The outline of the article is as follows: in \Cref{se: mass_scaling}, we first formally introduce mass scaling as an eigenvalue perturbation problem and also recall the terminology adopted within the community. Our analysis then focuses on global (\Cref{se: global_mass_scaling}) and local (\Cref{se: local_mass_scaling}) techniques. In both sections, we first recall some general results in eigenvalue perturbation theory before establishing rigorous bounds on the eigenvalues, step size, and condition number. Our bounds prove some of the observations stated by the original authors and support their numerical findings. Finally, \Cref{se: numerical_experiments} presents a few numerical experiments highlighting the sharpness of our bounds and provides further insights on the behavior of the methods. Lastly, \Cref{se: conclusion} summarizes our findings and suggests improvements for future work.

\section{Mass scaling}
\label{se: mass_scaling}
Mass scaling consists in perturbing the mass matrix in order to drive down the largest generalized eigenvalues constraining the critical time step of explicit time integration schemes, while ideally leaving the smallest ones unaffected. In general, the scaled mass matrix $\overline{M}$ is obtained by adding to the (lumped) mass matrix $M$ a well chosen symmetric positive semidefinite perturbation $E$:
\begin{equation*}
    \overline{M}=M+E.
\end{equation*}
Scaled quantities are often denoted with an overline and we will stick to this convention. Since we are exclusively dealing with real quantities, it will not be misidentified with complex conjugacy. Oftentimes, the mass matrix $M$ is lumped \citep{olovsson2005selective,tkachuk2013variational,cocchetti2013selective,oesterle2022intrinsically} with any suitable strategy such as the row-sum technique \cite{hughes2012finite} or the HRZ (Hinton-Rock-Zienkiewicz or diagonal scaling) method \cite{hinton1976note}. The scaling matrix $E$ is sometimes only defined locally, for each finite element, and the global matrix is then obtained from the usual assembly of element contributions. The success of mass scaling lies in the definition of $E$, which often depends on the problem and its discretization. The strategies proposed commonly fall in either one of the following categories:
\begin{itemize}
    \item \emph{Conventional mass scaling} (CMS), also sometimes called \emph{regular mass scaling}, only scales the diagonal entries of the mass matrix. It is called \emph{uniform} when all diagonal entries are scaled uniformly. Although it preserves the diagonal structure of the lumped mass matrix, it also scales down \emph{all} eigenvalues and adversely effects the dynamics, unless performed locally on critical elements within non-uniform meshes that severely constrain the step size \citep{olovsson2004selective,oesterle2022intrinsically,oesterle2023finite,hoffmann2023finite,cocchetti2013selective,cocchetti2015selective}. CMS is also typically used for scaling rotational degrees of freedom of beam or shell finite element formulations \cite{key1971transient,hughes1978reduced}. Despite being oftentimes inaccurate, this strategy is still widely used in commercial software packages, such as LS-DYNA \cite{hallquist2005ls}. Nevertheless, as shown in \cite{oesterle2022intrinsically,krauss2024intrinsically}, simply reformulating the original model may sometimes improve the accuracy.
    \item \emph{Selective mass scaling} (SMS), instead, selectively scales down the highest eigenfrequencies whose mode shapes are oftentimes described as structurally ``irrelevant'' \cite{oesterle2022intrinsically,oesterle2023finite}. Unfortunately, SMS usually leads to non-diagonal scaled mass matrices and requires solving a linear system at each time step. Not only does it ruin the fundamental concept of explicit time integration but it also requires a careful assessment of the tradeoff between the reduction of the number of iterations and the increased cost per iteration. Yet, the most successful mass scaling strategies fall in this second category and will be at the center of our attention in the forthcoming discussion.
\end{itemize}

As mentioned earlier, mass scaling strategies may be defined globally or locally. Global strategies often have appealing spectral properties but are rather impractical. Local strategies, instead, are more convenient and are usually favored for $C^0$ finite element discretizations, especially for beam, plate and shell formulations featuring both translational and rotational degrees of freedom (see e.g. \cite{key1971transient,hughes1978reduced,belytschko1980flexural}). However, their analysis is more tedious. In the sequel, we will distinguish these two cases as their treatment is fundamentally different.

\subsection{Global mass scaling}
\label{se: global_mass_scaling}
\subsubsection{Preliminaries}
In this section, we assume that the perturbation $E \in \mathbb{R}^{n \times n}$ is defined globally; i.e. its construction is oblivious to the underlying finite element method. Mass scaling, in the most general and abstract terms, may be seen as an eigenvalue perturbation problem. That is, given a symmetric matrix pair $(A,B)$ with $B$ positive definite, we must analyze how far off are the eigenvalues of $(A,B+E)$ with respect to those of $(A,B)$. The perturbation theory for generalized eigenproblems is the object of a rich literature, which has flourished ever since the groundbreaking work of Stewart in the late 1970s \cite{stewart1979pertubation}. However, most of the perturbation bounds are ill-suited for describing the properties of SMS since they bound some error measure of the eigenvalues \emph{uniformly} (i.e. independently of the eigenvalue number) whereas SMS predominantly perturbs selected eigenvalues. Qualitatively speaking though, the bounds indicate that small eigenvalues are less sensitive to perturbations than larger ones and might provide some early insight.

Instead of treating $E$ as some random perturbation, we will review some of the actual strategies proposed in the literature and position them in a more general context. It turns out none of methods reviewed in this section are fundamentally new. As a matter of fact, their theoretical foundations were already laid out decades ago \citep{stewart1979pertubation,stewart1990matrix,parlett1991symmetric}. We start by providing some definitions and recalling a few fundamental properties of generalized eigenproblems. Throughout this article, we consider the set of pairs of matrices $V := \mathbb{R}^{n \times n} \times \mathbb{R}^{n \times n}$, endowed with an obvious vector space structure:
\begin{itemize}
    \item $\forall (A,B),(C,D) \in V$,
    \begin{equation*}
        (A,B)+(C,D)=(A+C,B+D).
    \end{equation*}
    \item $\forall (A,B) \in V$, $\forall \mu \in \mathbb{R}$,
    \begin{equation*}
        \mu (A,B) = (\mu A, \mu B).
    \end{equation*}
\end{itemize}

In other words, $V := \mathbb{R}^{n \times n} \times \mathbb{R}^{n \times n}$ may be identified with $\mathbb{R}^{n \times 2n}$, to which it is isomorphic. This identification allows performing operations on matrix pairs as if they were block matrices. Moreover, we will exclusively focus on symmetric pairs (i.e. pairs formed by symmetric matrices) and in this context there exists a natural (partial) ordering.

\begin{definition}[Loewner partial order]
For two symmetric matrices $A,B \in \mathbb{R}^{n \times n}$, we write $A \succeq B$ (respectively $A \succ B$) if $A-B$ is positive semidefinite (respectively positive definite).
\end{definition}

For convenience, we also define the sets of positive (semi-)definite matrices.
\begin{definition}
\label{def: 1_SPSD_SPD_matrices}
The sets of symmetric positive semidefinite (SPSD) and symmetric positive definite (SPD) matrices of size $n$ are defined, respectively, as
\begin{equation*}
    \mathcal{S}_n=\{B \in \mathbb{R}^{n \times n} \colon B=B^T, B \succeq 0\} \quad \text{and} \quad \mathcal{S}_n^+=\{B \in \mathbb{R}^{n \times n} \colon B=B^T, B \succ 0\}.
\end{equation*}
\end{definition}

The analysis conducted in this article pertains to generalized eigenvalues and eigenvectors of matrix pairs. A pair $(\lambda, \mathbf{u}) \in \mathbb{C} \times \mathbb{C}^n$ with $\mathbf{u} \neq \mathbf{0}$ is called a generalized eigenpair of $(A,B)$ if
\begin{equation*}
    A \mathbf{u} = \lambda B \mathbf{u}.
\end{equation*}
This definition is called the \emph{standard form} of the eigenvalue problem. In contract, the mathematical literature commonly adopts the \emph{cross-product form}, whereby a pair $(\alpha, \beta) \in \mathbb{C}^2$ is called an eigenvalue if there exists $\mathbf{u} \neq \mathbf{0}$ such that
\begin{equation*}
    \beta A \mathbf{u}=\alpha B \mathbf{u}.
\end{equation*}
This representation treats $A$ and $B$ symmetrically. The corresponding generalized eigenvalue in standard form is recovered as $\lambda = \frac{\alpha}{\beta}$ (if $\beta \neq 0$ and is infinite otherwise). Obviously, this representation is non-unique since $(\alpha,\beta)$ and $(t \alpha, t \beta)$ (with $t \neq 0$) refers to the same eigenvalue. Thus, rigorously speaking, one must identify an eigenvalue with an equivalence class (see e.g. \cite{stewart1990matrix}). A matrix pair is called \emph{regular} if $(\alpha,\beta) \neq (0,0)$. We will exclusively consider regular matrix pairs in this work. More importantly, contrary to standard eigenproblems, the generalized eigenvalues of symmetric pairs are in general complex (see e.g. \citep{parlett1998symmetric,saad2011numerical}). The next lemma provides sufficient conditions guaranteeing real (nonnegative) eigenvalues.

\begin{lemma}[{\citep[][Theorem VI.1.15]{stewart1990matrix}}]
\label{lem: classical_problem}
Let $(A,B) \in \mathcal{S}_n \times \mathcal{S}_n^+$. Then, all generalized eigenvalues of $(A,B)$ are real nonnegative and there exists an invertible matrix $U \in \mathbb{R}^{n \times n}$ such that
\begin{equation*}
    U^TAU=D, \qquad U^TBU=I,
\end{equation*}
where $D=\diag(\lambda_1, \dots, \lambda_n)$ is a real nonnegative diagonal matrix containing the eigenvalues.
\end{lemma}

It can easily be shown that all eigenvalues are additionally positive if $A \in \mathcal{S}_n^+$. This condition will always be fulfilled in the forthcoming analysis. To avoid potential confusion, we will often specify the matrix pair when referring to eigenvalues, which are commonly numbered in ascending algebraic order; i.e.
\begin{equation}
\label{eq: ordering}
    \lambda_1(A,B) \leq \lambda_2(A,B) \leq \dots \leq \lambda_n(A,B).
\end{equation}
The set of all eigenvalues (i.e. the spectrum) will be denoted $\Lambda(A,B)$. Furthermore, in our context, we define an \emph{eigenfrequency} as $\omega = \sqrt{\lambda}$. The next lemma is the main building block for mass scaling techniques.

\begin{lemma}[{\citep[][Corollary 3.6]{voet2024robust}}]
\label{lem: equivalence_conditions}
Let $A \in \mathcal{S}_n$, $B,\widetilde{B} \in \mathcal{S}_n^+$ and denote $E=\widetilde{B}-B$. Then the statements
\begin{enumerate}[noitemsep]
    \item $E \succeq 0$, \label{eq: eq_cond_1}
    \item $\widetilde{B} \succeq B$, \label{eq: eq_cond_2}
    \item $\Lambda(B,\widetilde{B}) \subset (0,1]$, \label{eq: eq_cond_3}
\end{enumerate}
are all equivalent and imply that $\lambda_k(A,\widetilde{B}) \leq \lambda_k(A,B)$ for all $k=1,\dots,n$.
\end{lemma}

Thus, in particular, positive semi-definiteness of the scaling matrix $E$ guarantees a decrease of the generalized eigenvalues. This is usually the first prerequisite in designing mass scaling techniques.

\subsubsection{Linear fractional transformations}
Modifying a matrix pair generally modifies its eigenvalues nontrivially. However, there exist special transformations for which the transformed eigenvalues are known explicitly. Linear Fractional Transformations are one of them.

\begin{definition}[Linear factional transformation]
\label{def: linear_frac_transform}
A linear application $L \colon V \to V$ defined as 
\begin{equation*}
    L(A,B) := (w_{11}A+w_{21}B, w_{12}A+w_{22}B)
\end{equation*}
where $w_{ij} \in \mathbb{R}$ for $i,j=1,2$ is called a \emph{linear fractional transformation} (LFT).
\end{definition}

LFTs are analogous to shift-and-invert strategies \citep{parlett1998symmetric,saad2011numerical} and are widely used within eigensolvers for accelerating the convergence to desired eigenvalues. In this context, they are better known as \textit{spectral transformations} \cite{ericsson1980spectral}. The upcoming lemma and subsequent remark will better explain the name given to this type of transformation.

\begin{remark}
As noted in \cite{stewart1990matrix}, by gathering the coefficients $w_{ij}$ in a matrix $W \in \mathbb{R}^{2 \times 2}$, the application may be compactly expressed as
\begin{equation*}
    L(A,B)=(A,B)(W \otimes I_n)
\end{equation*}
and is entirely determined by the matrix $W$.
\end{remark}

\begin{definition}[Non-degenerate LFT]
An LFT is called \emph{degenerate} if $\det(W)=0$ and is called \emph{non-degenerate} otherwise.
\end{definition}

As shown in the following technical lemma \citep[][Theorem VI.1.6]{stewart1990matrix} and \citep[][Theorem 9.1]{saad2011numerical}, LFTs do not affect the eigenvectors and transform the eigenvalues in a simple manner. We include its proof, which is an exercise in \cite{stewart1990matrix}.

\begin{lemma}
\label{lem: matrix_pencil_translation}
Let $(A,B)$ be a regular matrix pencil and $L(A,B)$ be a non-degenerate LFT. Then,
\begin{itemize}[noitemsep]
    \item The eigenvectors of $(A,B)$ and $L(A,B)$ are the same.
    \item The eigenvalues $(\alpha,\beta)$ of $(A,B)$ are related to the eigenvalues $(\overline{\alpha},\overline{\beta})$ of $L(A,B)$ through the relation
    \begin{equation*}
    (\overline{\alpha},\overline{\beta}) = (\alpha,\beta)W.
\end{equation*}
\end{itemize}
\end{lemma}
\begin{proof}
Let $(\alpha,\beta)$ be an eigenvalue of $(A,B)$ associated to an eigenvector $\mathbf{u}$. The eigenvalue equation $\beta A \mathbf{u}=\alpha B \mathbf{u}$ can be expressed as $(A,B)(\mathbf{v} \otimes \mathbf{u})=0$ where $\mathbf{v}^T=(\beta, -\alpha)$. Now we note that
\begin{equation*}
    (A,B)(\mathbf{v} \otimes \mathbf{u})=(A,B)(WW^{-1} \otimes I_n)(\mathbf{v} \otimes \mathbf{u})=(A,B)(W \otimes I_n)(W^{-1}\mathbf{v} \otimes \mathbf{u})=L(A,B)(\overline{\mathbf{v}} \otimes \mathbf{u})
\end{equation*}
where we have defined $\overline{\mathbf{v}}=W^{-1}\mathbf{v}$. Thus, $\mathbf{u}$ is also an eigenvector of $L(A,B)$ and its associated eigenvalue is determined from $(\overline{\beta},-\overline{\alpha})=(\beta,-\alpha)W^{-T}$. The statement follows after rewriting and simplifying this relation.
\end{proof}

\Cref{lem: matrix_pencil_translation} shows that there is a one-to-one relation between the eigenvalues of $(A,B)$ and $L(A,B)$ via the invertible matrix $W$. The name \emph{linear fractional transformation} now takes on its full meaning since the transformed eigenvalues (in standard form) are given by
\begin{equation*}
    \overline{\lambda}=\frac{\overline{\alpha}}{\overline{\beta}}=\frac{w_{11}\alpha+w_{21}\beta}{w_{12}\alpha+w_{22}\beta}=\frac{w_{11}\lambda+w_{21}}{w_{12}\lambda+w_{22}}=f(\lambda).
\end{equation*}
On its interval of definition, one easily shows that $f(\lambda)$ is a strictly increasing (resp. strictly decreasing) function of $\lambda$ if $\det(W) > 0$ (resp. $\det(W) < 0$). Thus, the former preserves the eigenvalue numbering. Among the mass scaling strategies proposed in the literature, two of them are LFTs:

\begin{itemize}
    \item Uniform mass scaling is an LFT with
    \begin{equation*}
        W=
    \begin{pmatrix}
        1 & 0 \\
        0 & \mu
    \end{pmatrix}
    \iff \overline{\lambda}=\frac{\lambda}{\mu}.
    \end{equation*}
\item \emph{Stiffness proportional} SMS \citep{macek1995mass,olovsson2005selective} in an LFT with
    \begin{equation*}
    W=
    \begin{pmatrix}
        1 & \mu \\
        0 & 1
    \end{pmatrix}
    \iff \overline{\lambda}=\frac{\lambda}{\mu\lambda+1}.
    \end{equation*}
\end{itemize}
Uniform mass scaling simply scales all eigenvalues by $\mu > 0$ and therefore also affects the smallest ones. This impacts the accuracy of the dynamics, as testified by the numerical experiments in \cite{olovsson2004selective}. On the contrary, if the scaling parameter is wisely chosen, stiffness proportional mass scaling allows to selectively scale down the largest eigenvalues while nearly preserving the smallest ones. Olovsson et al. \cite{olovsson2005selective} recommend choosing $\mu=10^k/\lambda_n$ with $k \in \mathbb{N}$. If $\lambda_n/\lambda_1 \gg 10^{k}$, the smallest eigenvalues are barely affected while the largest ones are roughly scaled down by a factor of $10^k$. In practice, computing $\mu$ only requires a (coarse) estimation for $\lambda_n$. Unfortunately, the pleasing property of stiffness proportional SMS comes at the price of a non-diagonal scaled mass matrix $\overline{M}=M+\mu K$. In fact, if one uses this scaled mass matrix in an explicit scheme for linear elasto-dynamics, he might as well consider using an implicit unconditionally stable Newmark method from the very start, whose workload per iteration is similar. The method is even less suited to nonlinear problems undergoing large rotations and deformations as the stiffness matrix changes during the course of the simulation and is typically only accessible through matrix-vector products. Stiffness proportional SMS would require reassembling the stiffness matrix at each iteration, which is prohibitively expensive and is the main reason why, according to some authors \citep{oesterle2023finite,tkachuk2015direct}, the method has not been implemented in commercial software packages. Nevertheless, according to Askes et al. \cite{askes2011increasing}, stiffness proportional SMS is employed in other branches of mechanics, albeit under different names.

Instead of working with the original matrix pair $(A,B)$, if $B$ is invertible, one might consider the equivalent pair $(B^{-1}A,I)$, which shares the same eigenvalues and eigenvectors as $(A,B)$. This reduces the generalized eigenvalue problem to standard form and enabled the authors in \cite{tkachuk2014local} to generalize the LFT to rational polynomials and even virtually any function. Given a matrix $A \in \mathbb{C}^{n \times n}$ and a function $f \colon \mathbb{C} \to \mathbb{C}$ well defined on the spectrum of $A$, it is well-known that if $(\lambda, \mathbf{u})$ is an eigenpair of $A$, then $(f(\lambda),\mathbf{u})$ is an eigenpair of $f(A)$ \cite{higham2008functions}. In \cite{tkachuk2014local} the authors applied this general result to $M^{-1}K$ with a function
\begin{equation*}
    f(t) = \frac{t}{1+p(t)}
\end{equation*}
where $p(t) \in \mathbb{R}_d[t]$ is a real coefficient polynomial of degree $d$. After rewriting $f(M^{-1}K)$ as a matrix pair, they considered using $(K,M(I + p(M^{-1}K)))$ with the second degree polynomial $p(t)=ct^2$ for some coefficient $c>0$. Although the method (refered to therein as \emph{polynomial matrix} SMS) generalizes uniform and stiffness proportional SMS to arbitrary polynomials, it suffers from the same drawbacks as stiffness proportional SMS. Moreover, it explicitly features the inverse of $M$ for any degree $d \geq 2$ and is therefore practically limited to diagonal (lumped) mass matrices. We are not aware if any further research alleviating this issue.

\subsubsection{Global deflation}
\label{se: global_deflation}
Instead of transforming the entire spectrum, it is also possible to modify only a few selected eigenvalues. In numerical linear algebra, this procedure is commonly known as \emph{deflation}. Its origin dates back to the early work of Hotelling \cite{hotelling1943some} and is well-known for standard eigenproblems. For generalized eigenproblems, the procedure is described e.g. in \cite{saad2011numerical}. For $C^0$ finite element discretizations, González and Park \cite{gonzalez2020large} revisited this idea and improved the earlier work of Tkachuk and Bischoff \cite{tkachuk2014local} by proposing a fast solver for the scaled mass matrix based on the Woodbury matrix identity. While the method remains impractical for $C^0$ discretizations due to the sheer number of inaccurate frequencies, it becomes much more appealing for maximally smooth isogeometric discretizations featuring far fewer inaccurate frequencies \cite{cottrell2006isogeometric,hughes2014finite,voet2024robust}. It is briefly recalled here.

\begin{definition}[Scaled matrix pair]
\label{def: deflated_pencil}
Let $(A,B) \in \mathcal{S}_n \times \mathcal{S}_n^+$ and $f$, $g$ be two functions defined on the spectrum of $(A,B)$. The scaled matrix pair $(\overline{A},\overline{B})$ is defined as
\begin{align*}
    \overline{A} &= A+Vf(D_2)V^T, \\
    \overline{B} &= B+Vg(D_2)V^T,
\end{align*}
where $V=BU_2 \in \mathbb{R}^{n \times r}$, with $U_2=[\mathbf{u}_{n-r+1}, \dots, \mathbf{u}_n]$ the matrix formed by the last $r$ $B$-orthonormal eigenvectors of $(A,B)$ and $D_2=\diag(\lambda_{n-r+1}, \dots, \lambda_n) \in \mathbb{R}^{r \times r}$ the diagonal matrix formed by the last $r$ eigenvalues with $r \ll n$.    
\end{definition}

\begin{theorem}[{\citep[][Theorem 4.3]{voet2024robust}}]
\label{th: low_rank_pert_AB}
Let $(A,B) \in \mathcal{S}_n \times \mathcal{S}_n^+$ and $(\overline{A},\overline{B})$ be the scaled matrix pair introduced in \Cref{def: deflated_pencil}. Then,
\begin{itemize}[noitemsep]
    \item The eigenvectors of $(A,B)$ and $(\overline{A},\overline{B})$ are the same.
    \item The eigenvalues of $(\overline{A},\overline{B})$ are given by:
    \begin{equation*}
    \overline{\lambda}_{i_k}=
    \begin{cases}
    \lambda_k & \text{ for } k=1,\dots,n-r, \\
    \frac{\lambda_k+f(\lambda_k)}{1+g(\lambda_k)} & \text{ for } k=n-r+1,\dots,n.
    \end{cases}
    \end{equation*}
\end{itemize}
\end{theorem}
In the context of mass scaling, one typically sets
\begin{equation}
\label{eq: mass_scaling}
    f(\lambda)=0 \quad \text{and} \quad g(\lambda)=\frac{\lambda}{\lambda_{n-r}}-1,
\end{equation}
and the transformed eigenvalues are then given by
\begin{equation*}
    \overline{\lambda}_{k}=
    \begin{cases}
    \lambda_k & \text{ for } k=1,\dots,n-r, \\
    \lambda_{n-r} & \text{ for } k=n-r+1,\dots,n.
    \end{cases}
\end{equation*}
The choices of $f$ and $g$ are not unique and one may equivalently choose
\begin{equation*}
    f(\lambda)=\lambda_{n-r}-\lambda \quad \text{and} \quad g(\lambda)=0,
\end{equation*}
which would instead scale the stiffness matrix. Optionally, $\lambda_{n-r}$ can be substituted with a cutoff value, as suggested in \cite{tkachuk2014local,gonzalez2020large}. However, choosing $\lambda_{n-r}$ preserves the eigenvalue ordering and graphically shaves off the upper part of the spectrum (see Figure \ref{fig: truncation}). The increase of the critical time step for the central difference method is directly appreciated since
\begin{equation*}
    \frac{\overline{\Delta t}_c}{\Delta t_c}=\sqrt{\frac{\lambda_n}{\lambda_{n-r}}}.
\end{equation*}

\begin{figure}[H]
    \centering
    \includegraphics[width=0.6\textwidth]{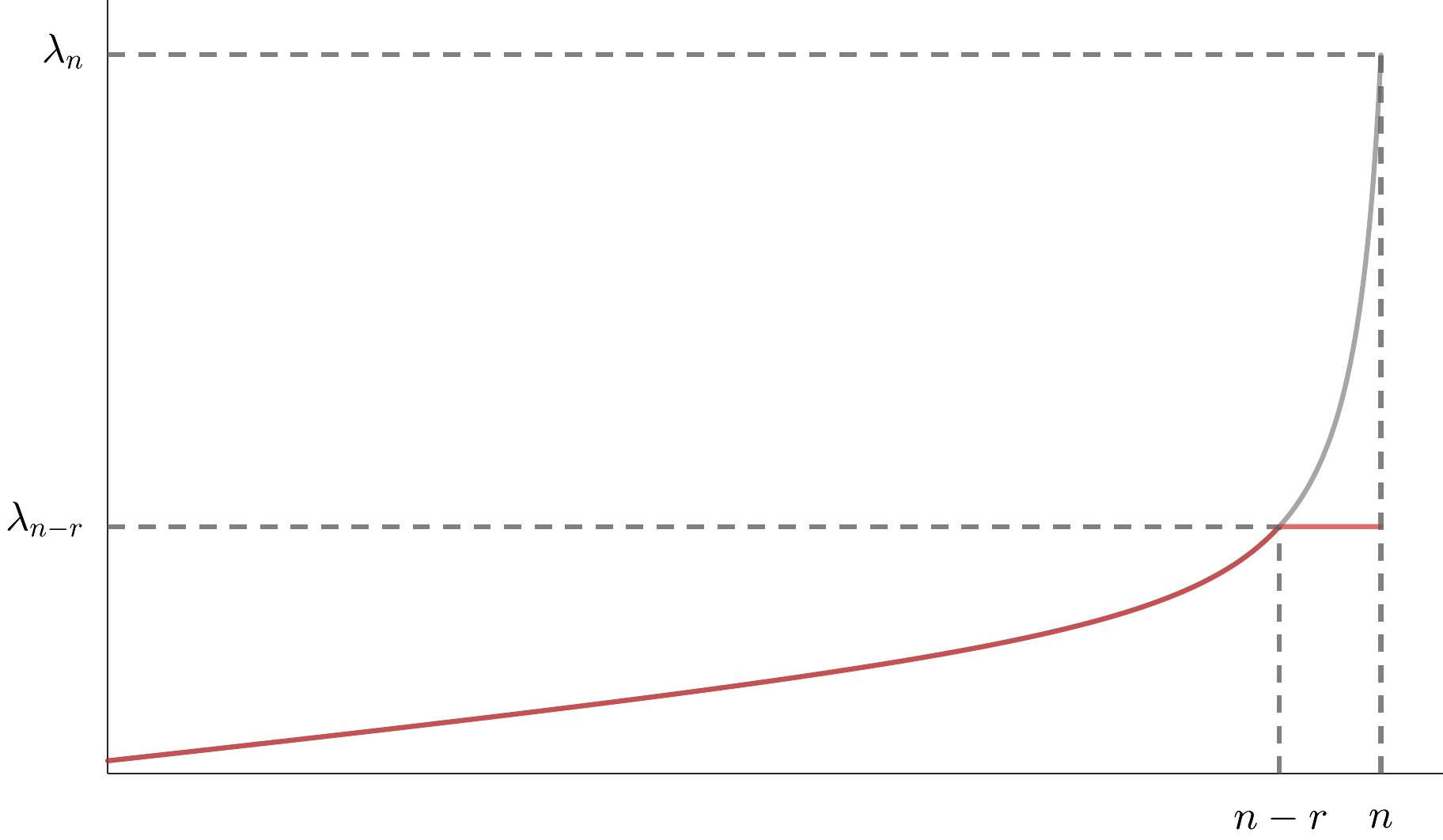}
    \caption{Truncation of the largest eigenvalues}
    \label{fig: truncation}
\end{figure}

Although the scaled mass matrix is non-diagonal, \citep[][Equation (28)]{gonzalez2020large} provides an explicit expression for its inverse, based on the Woodbury matrix identity. Unfortunately, this strategy still requires computing some of the largest eigenpairs of a (generally very large) matrix pair, which is prohibitively expensive for $C^0$ finite element discretizations due to the ``optical branches'' of the spectrum contributing too many inaccurate high frequencies. The method becomes much more affordable and even realistic for maximally smooth isogeometric discretizations featuring far fewer inaccurate high frequencies \citep[][Theorems A.1 and A.2]{voet2024robust} whose separation from the rest of the spectrum is also a key advantage in accelerating eigensolvers based on Krylov subspaces. Alternatively, deflation strategies may rely on low-frequency modes combined with a projector in the orthogonal complement \cite{gonzalez2020large}. Although such methods do not offer any computational advantages in this setting, they form the cornerstone of inexact ad hoc deflation strategies reviewed in \Cref{se: olovsson,se: hoffmann}.

\Cref{th: low_rank_pert_AB} shows that the eigenspaces of the original and scaled matrix pairs defined in \Cref{def: deflated_pencil} are exactly the same. For more general mass scaling strategies, they are usually different although they might intersect. The next lemma recalls an elementary albeit important fact. 

\begin{lemma}
\label{lem: common_eigenvector}
If $(\lambda, \mathbf{u})$ and $(\mu, \mathbf{u})$ are eigenpairs of $(A,B)$ and $(B,C)$, respectively, then $(\lambda \mu, \mathbf{u})$ is an eigenpair of $(A,C)$.  
\end{lemma}
\begin{proof}
If $(\lambda, \mathbf{u})$ and $(\mu, \mathbf{u})$ are eigenpairs of $(A,B)$ and $(B,C)$, respectively, then
\begin{equation*}
    A\mathbf{u}=\lambda B \mathbf{u} = \lambda \mu C \mathbf{u}
\end{equation*}
showing that $\lambda \mu$ is an eigenvalue of $(A,C)$.
\end{proof}
The previous result remains practically relevant in case the eigenspaces nearly intersect. This is also probably the main reason underpinning the success of mass lumping strategies, which tend to well approximate low-frequency eigenmodes. Note that if a vector $\mathbf{u}$ is known to be an eigenvector of $(A,B)$ and $(B,C)$, then the corresponding eigenvalue of $(A,C)$ may be directly computed as
\begin{equation*}
    \lambda(A,C)=R(\mathbf{u})Q(\mathbf{u}),
\end{equation*}
where
\begin{equation*}
    R(\mathbf{x})=\frac{\mathbf{x}^T A \mathbf{x}}{\mathbf{x}^T B \mathbf{x}} \quad \text{and} \quad Q(\mathbf{x})=\frac{\mathbf{x}^T B \mathbf{x}}{\mathbf{x}^T C \mathbf{x}}
\end{equation*}
are the Rayleigh quotients.

\subsection{Local mass scaling}
\label{se: local_mass_scaling}
\subsubsection{Preliminaries}
Some of the global methods we have previously reviewed also have a local counterpart. The two approaches rarely coincide, expect for uniform or stiffness proportional SMS, provided the scaling parameter is chosen uniformly across elements. In general though, fully assembled system matrices do not immediately inherit local properties from element contributions. Nevertheless, there often exist very useful connections. We will first recall some of them and later analyze a few practical schemes proposed in the literature.

\begin{theorem}[{\citep[][Theorem 3.4]{voet2023mathematical}}]
\label{th: eig_bounds}
Let $A \in \mathcal{S}_n$, $B,C \in \mathcal{S}_n^+$ and let all eigenvalues be numbered in ascending algebraic order. Then
\begin{subequations}
\begin{align}
    \lambda_k(A,C)\lambda_1(C,B) &\leq \lambda_k(A,B) \leq \lambda_k(A,C)\lambda_n(C,B) \qquad 1 \leq k \leq n, \label{eq: ineq1}\\
    \lambda_1(A,C)\lambda_k(C,B) &\leq \lambda_k(A,B) \leq \lambda_n(A,C)\lambda_k(C,B) \qquad 1 \leq k \leq n. \label{eq: ineq2}
\end{align}
\end{subequations}
\end{theorem}

A direct application of Theorem \ref{th: eig_bounds}, inequalities \eqref{eq: ineq1}, with $A=K$, $B=M$ and $C=\overline{M}$ leads to
\begin{equation}
\label{eq: eig_pert_bounds_mass}
   \lambda_1(\overline{M},M) \leq \frac{\lambda_k(K,M)}{\lambda_k(K,\overline{M})} \leq \lambda_n(\overline{M},M) \qquad 1 \leq k \leq n. 
\end{equation}
This inequality will play a central role for deriving time step estimates. As we have discussed, bounds on the condition number of the scaled mass matrix are also desirable if linear systems with the latter are solved iteratively. We recall that the spectral condition number of a symmetric positive definite matrix $A$ is defined as $\kappa(A)=\lambda_n(A)/\lambda_1(A)$. Analogously, we define the condition number of a pair of symmetric positive definite matrices $(A,B)$ as $\kappa(A,B)=\lambda_n(A,B)/\lambda_1(A,B)$. The next corollary provides a useful inequality relating the two.

\begin{corollary}
\label{cor: conditioning_bound}
Let $A,B \in \mathcal{S}_n^+$ and let $\kappa$ denote the spectral condition number. Then
\begin{equation*}
    \frac{\kappa(A)}{\kappa(B)} \leq \kappa(A,B).
\end{equation*}
\end{corollary}
\begin{proof}
Applying Theorem \ref{th: eig_bounds}, inequalities \eqref{eq: ineq1} with $C=I$ yields
\begin{equation*}
    \lambda_1(A,B) \leq \frac{\lambda_1(A)}{\lambda_1(B)} \quad \text{and} \quad \frac{\lambda_n(A)}{\lambda_n(B)} \leq \lambda_n(A,B).
\end{equation*}
Combining the two previous inequalities concludes the proof.
\end{proof}

So far, all inequalities pertain to the global system matrices. Yet, the properties of finite element matrices are usually related to those of the element matrices from which they are formed. The two are connected through the assembly procedure, which we describe next. To simplify the presentation, we assume that the mesh is made up of a single type of element such that all element matrices have the same size $m$. Then, given a set of element matrices $\{A_e\}_{e=1}^N \subseteq \mathcal{S}_m$, where $N$ is the number of elements, the global matrix $A \in \mathbb{R}^{n \times n}$ can be expressed as 
\begin{equation}
\label{eq: assembly}
    A=\sum_{e=1}^N L_e^TA_eL_e,
\end{equation}
where $L_e^T = [\mathbf{e}_{i_1},\dots,\mathbf{e}_{i_m}] \in \mathbb{R}^{n \times m}$ contains a subset of the columns of the identity matrix $I_n$ with indices $i_k$ for $k=1,\dots,m$ depending on the connectivity of the element. Denoting $\mathsf{L}^T=[L_1^T, \dots, L_N^T] \in \mathbb{R}^{n \times mN}$ and $\mathsf{A}=\diag(A_1, \dots, A_N) \in \mathbb{R}^{mN \times mN}$, \cref{eq: assembly} is compactly written as $A=\mathsf{L}^T\mathsf{A}\mathsf{L}$. With this expression at hand, we now summarize some of the main results of \citep{fried1972bounds,wathen1987realistic}. The proofs are included and will be useful later in the analysis.

\begin{lemma}[{\cite[][Theorem 1]{fried1972bounds}}]
\label{lem: conditioning_FEM_matrices}
Given a finite element matrix $A=\mathsf{L}^T\mathsf{A}\mathsf{L}$. Then
\begin{align*}
    \max_e \lambda_m(A_e) &\leq \lambda_n(A) \leq p_{\max} \max_e \lambda_m(A_e), \\
    \min_e \lambda_1(A_e) &\leq \lambda_1(A),
\end{align*}
where $p_{\max} \in \mathbb{N}^*$ is the maximum number of elements to which a node is connected.
\end{lemma}
\begin{proof}
Given $\mathbf{x} \in \mathbb{R}^n$ with $\|\mathbf{x}\|_2=1$,
\begin{equation*}
    \lambda_{1}(\mathsf{A})\|\mathsf{L}\mathbf{x}\|_2^2 \leq \mathbf{x}^T A \mathbf{x} =\mathbf{x}^T\mathsf{L}^T\mathsf{A}\mathsf{L}\mathbf{x} \leq \lambda_{mN}(\mathsf{A})\|\mathsf{L}\mathbf{x}\|_2^2.
\end{equation*}
Since $\mathsf{A}$ is block-diagonal, $\lambda_{1}(\mathsf{A})=\min_e \lambda_1(A_e)$ and $\lambda_{mN}(\mathsf{A})=\max_e \lambda_m(A_e)$. Moreover, 
\begin{equation*}   \|\mathsf{L}\mathbf{x}\|_2^2=\mathbf{x}^T\mathsf{L}^T\mathsf{L}\mathbf{x}=\sum_{e=1}^N \mathbf{x}^TL_e^TL_e\mathbf{x}=\sum_{e=1}^N \|\mathbf{x}_e\|_2^2 = \sum_{i=1}^n p_i x_i^2
\end{equation*}
where $\mathbf{x}_e=L_e\mathbf{x}$ is a subvector of $\mathbf{x}$ and $p_i$ is the number of elements connected to node $i$. Since $1 \leq p_i \leq p_{\max}$ for all $i=1,\dots,n$ and $\|\mathbf{x}\|_2=1$, we immediately deduce that $1 \leq \|\mathsf{L}\mathbf{x}\|_2^2 \leq p_{\max}$. Combining these results, we obtain
\begin{equation*}
    \min_e \lambda_1(A_e) \leq \mathbf{x}^T A\mathbf{x} \leq p_{\max} \max_e \lambda_m(A_e),
\end{equation*}
which yields an upper bound on $\lambda_n(A)$ and a lower bound on $\lambda_1(A)$. The lower bound of $\lambda_n(A)$ is deduced for a specific choice of vector. Namely, we choose $\mathbf{x}$ to have zero components everywhere except for the subvector $\mathbf{x}_e$ corresponding to the element with largest eigenvalue, which we call $l$. Then, $\|\mathbf{x}\|_2=\|\mathbf{x}_l\|_2=1$ and by choosing $\mathbf{x}_l$ to be the eigenvector associated to the largest eigenvalue of $A_l$, we finally obtain
\begin{equation*}
    \lambda_n(A) \geq \mathbf{x}^T A \mathbf{x} = \sum_{e=1}^N \mathbf{x}_e^T A_e \mathbf{x}_e \geq \mathbf{x}_l^T A_l \mathbf{x}_l= \max_e \lambda_m(A_e),
\end{equation*}
where the second inequality follows from the positive semi-definiteness of $A_e$ for $e=1,\dots,N$.
\end{proof}
The previous lemma can be used to compute an upper bound on the condition number of the mass matrix as \citep[][Corollary 1]{fried1972bounds}
\begin{equation}
\label{eq: upper_cond_M}
1 \leq \kappa(M) \leq p_{\max} \frac{\max_e \lambda_m(M_e)}{\min_e \lambda_1(M_e)}.
\end{equation}
For immersed finite element discretizations, the condition number may become unbounded as $\lambda_1(M_e)$ goes to zero for arbitrarily small cut elements.

The next lemma, attributed to Irons and Treharne \cite{irons1971bound}, has already appeared numerous times in the finite element literature (see e.g. \citep{wathen1987realistic,hughes2012finite}). Its short and elegant proof, due to Wathen \cite{wathen1987realistic}, is included for completeness.

\begin{theorem}[{\cite{irons1971bound}}]
\label{th: bounds_generalized_eig}
Let $\{A_e\}_{e=1}^N \subseteq \mathcal{S}_m$ and $\{B_e\}_{e=1}^N \subseteq \mathcal{S}_m^+$ be sets of element matrices for global finite element matrices $A$ and $B$, respectively, and $N$ be the number of elements. Then
\begin{equation*}
    \min_e \lambda_1(A_e,B_e) \leq \lambda_1(A,B), \qquad \lambda_n(A,B) \leq \max_e \lambda_m(A_e,B_e).
\end{equation*}
\end{theorem}
\begin{proof}
\begin{equation*}
    \lambda_1(A,B)=\min_{\substack{ \mathbf{x} \in \mathbb{R}^n \\ \mathbf{x} \neq \mathbf{0}}} \frac{\mathbf{x}^T \mathsf{L}^T\mathsf{A}\mathsf{L}\mathbf{x}}{\mathbf{x}^T \mathsf{L}^T\mathsf{B}\mathsf{L} \mathbf{x}}=\min_{\substack{ \mathbf{y} \in \mathcal{Y} \\ \mathbf{y} \neq \mathbf{0}}} \frac{\mathbf{y}^T \mathsf{A} \mathbf{y}}{\mathbf{y}^T \mathsf{B} \mathbf{y}} \geq \min_{\substack{ \mathbf{y} \in \mathbb{R}^{mN} \\ \mathbf{y} \neq \mathbf{0}}} \frac{\mathbf{y}^T \mathsf{A} \mathbf{y}}{\mathbf{y}^T \mathsf{B} \mathbf{y}}=\min_e \lambda_1(A_e,B_e).
\end{equation*}
where $\mathcal{Y} \subseteq \mathbb{R}^{mN}$ is the space spanned by the columns of $\mathsf{L}$. Similarly,
\begin{equation*}
    \lambda_n(A,B)=\max_{\substack{\mathbf{x} \in \mathbb{R}^n \\ \mathbf{x} \neq \mathbf{0}}} \frac{\mathbf{x}^T \mathsf{L}^T\mathsf{A}\mathsf{L}\mathbf{x}}{\mathbf{x}^T \mathsf{L}^T\mathsf{B}\mathsf{L} \mathbf{x}}=\max_{\substack{ \mathbf{y} \in \mathcal{Y} \\ \mathbf{y} \neq \mathbf{0}}} \frac{\mathbf{y}^T \mathsf{A} \mathbf{y}}{\mathbf{y}^T \mathsf{B} \mathbf{y}} \leq \max_{\substack{ \mathbf{y} \in \mathbb{R}^{mN} \\ \mathbf{y} \neq \mathbf{0}}} \frac{\mathbf{y}^T \mathsf{A} \mathbf{y}}{\mathbf{y}^T \mathsf{B} \mathbf{y}}=\max_e \lambda_m(A_e,B_e).
\end{equation*}
\end{proof}

The upper bound of \Cref{th: bounds_generalized_eig} is usually tight \cite{cottereau2018stability} and translates into a relatively large and conservative step size estimate for applications in explicit dynamics. However, neither \Cref{th: bounds_generalized_eig} nor \Cref{lem: conditioning_FEM_matrices} are applicable for bounding the eigenvalues of $CK$, where $C$ is a reciprocal mass matrix (i.e. a sparse approximate inverse of the mass matrix) \citep{tkachuk2015direct,schaeuble2017variationally}. Indeed, $CK$ does not have the form described in \cref{eq: assembly} and is generally not even symmetric. Instead, upper bounds based on the classical Gershgorin circle theorem \cite{varga2011gersgorin} or its generalization given by Ostrowski’s bound \cite{tkachuk2019time} have been suggested. In \cite{cocchetti2013selective}, the authors have noted that the former produced large overestimates for the element largest eigenfrequency when using $M^{-1}K$. This is not surprising given that pre-multiplication of $K$ with the diagonal matrix $M^{-1}$ will only rescale its rows. Substituting $M^{-1}K$ with the similar (and symmetric) matrix $M^{-1/2}KM^{-1/2}$ sometimes produces sharper bounds \cite{schaeuble2018time}. Combining the Gershgorin theorem with a diagonal similarity transformation might also improve the bounds \cite[][Corollary 6.1.6]{horn2012matrix}, particularly for beam, plate or shell problems, when the matrix entries have different physical units \cite{tkachuk2019time}.

In contrast, several authors in the engineering community have undertaken the task of deriving analytical or heuristic estimates for the largest element frequencies (see e.g. \cite{cocchetti2013selective,cocchetti2015selective,hartmann2015mass,adam2015stable}). However, these estimates are not necessarily conservative and rarely account for common changes to element formulations due, for instance, to reduced integration, locking-free formulations or the addition of penalty terms \citep{leidinger2019explicit,stoter2023critical}. Nevertheless, they may help identify critical elements \cite{hartmann2015mass}, before using refined estimates.

As a matter of fact, the sharpness of the upper bound of \Cref{th: bounds_generalized_eig} also highlights how individual elements might negatively impact the critical time step and is an incentive for modifying the element formulations prior to assembly. This reasoning drove early developments of local mass scaling techniques \cite{olovsson2004selective}. We are now ready to analyze some of the practical schemes proposed in the literature.

\subsubsection{Conventional mass scaling}
Conventional mass scaling only scales the diagonal entries of local (lumped) mass matrices. Without loss of generality, we may assume only the first $r$ entries are scaled by a factor $\alpha>1$ such that if $M_e=\diag(d_1,\dots,d_m)$, then $\overline{M}_e = \diag(\alpha d_1, \dots ,\alpha d_r, d_{r+1},\dots,d_m)$. In this case, the following corollary holds.

\begin{corollary}
\label{cor: CMS}
For conventional mass scaling with a uniform scaling parameter $\alpha>1$ the global eigenfrequencies $\overline{\omega}_i$ and $\omega_i$ of the scaled and unscaled system, respectively, satisfy
\begin{equation*}
    1 \leq \frac{\omega_i}{\overline{\omega}_i} \leq \sqrt{\alpha}.
\end{equation*}
\end{corollary}
\begin{proof}
Combining \Cref{th: eig_bounds,th: bounds_generalized_eig},
\begin{equation*}
     \min_e \lambda_1(\overline{M}_e, M_e) \leq \lambda_1(\overline{M},M) \leq \frac{\lambda_i(K,M)}{\lambda_i(K, \overline{M})} \leq \lambda_n(\overline{M},M) \leq \max_e \lambda_m(\overline{M}_e, M_e).
\end{equation*}    
Since only selected entries of the element lumped mass matrices $M_e$ are scaled by $\alpha>1$, $\lambda_1(\overline{M}_e, M_e)=1$ and $\lambda_m(\overline{M}_e, M_e)=\alpha$. The result then immediately follows after taking the square root.
\end{proof}

\Cref{cor: CMS} neither depends on the nature of the scaled degrees of freedom nor on the original physical model. Thus, it applies both to rotational mass scaling \cite{key1971transient,hughes1978reduced} and to the more recent intrinsically selective mass scaling technique (ISMS) \cite{oesterle2022intrinsically,krauss2024intrinsically}. While the sharpness of the bounds certainly depends on these factors, the numerical experiments reported in \cite{krauss2024intrinsically} for ISMS reveal that the upper bound is tight if $\alpha$ is not too large.

\subsubsection{Local deflation}
\label{se: local_deflation}
Several authors \cite{ye2017selective,gonzalez2020large} have suggested applying locally the eigenvalue deflation techniques reviewed in \Cref{se: global_deflation}. Following the construction mentioned in \cite{gonzalez2018inverse,gonzalez2020large}, all local mass matrices $\{M_e\}_{e=1}^N$ are scaled according to \Cref{def: deflated_pencil} with $f(\lambda)=0$ and $g(\lambda)=\alpha$ and their contributions are assembled in the usual manner. For simplicity, we assume a uniform rank $0<r<m$ and cutoff value $\alpha>0$ for all elements. For this construction, the following corollary holds.

\begin{corollary}
\label{cor: local_deflation_st1}
If all element matrices are scaled following \Cref{def: deflated_pencil} with $f(\lambda)=0$ and $g(\lambda)=\alpha$ for a uniform cutoff value of $\alpha$, then, the global eigenfrequencies $\overline{\omega}_i$ and $\omega_i$ of the scaled and unscaled system, respectively, satisfy
\begin{equation*}
    1 \leq \frac{\omega_i}{\overline{\omega}_i} \leq \sqrt{1+\alpha}.
\end{equation*}
\end{corollary}
\begin{proof}
The proof strategy is exactly the same as in \Cref{cor: CMS} and only requires computing the eigenvalues of $(\overline{M}_e, M_e)$. Applying \Cref{th: low_rank_pert_AB} locally (with $A=B=M_e$, $f(\lambda)=\alpha$ and $g(\lambda)=0$), we deduce that the eigenvalues of $(\overline{M}_e, M_e)$ are
\begin{equation*}
    \lambda_k(\overline{M}_e, M_e)=
    \begin{cases}
    1 & \text{ for } k=1,\dots,m-r, \\
    1+\alpha & \text{ for } k=m-r+1,\dots,m.
    \end{cases}
\end{equation*}
The result then immediately follows after taking the square root.
\end{proof}

\Cref{cor: CMS,cor: local_deflation_st1} are strikingly similar and may achieve a similar increase of the critical time step. However, local deflation strategies may deliver far greater accuracy by specifically targeting restrictive modes that contribute very little to the solution. Selecting the right modes is strongly problem-dependent. For instance, for thin-walled structures, they are usually thickness stretch modes \cite{cocchetti2013selective,cocchetti2015selective,hoffmann2023finite} while for nearly incompressible materials, they are volumetric modes \cite{ye2017selective}. However, if $\alpha$ becomes too large, other deformation modes may limit the critical time step and, as shown experimentally in \cite{gonzalez2018inverse,gonzalez2020large}, the method then loses significant accuracy by introducing high-frequency modes in the low-frequency spectrum. This phenomenon is less likely if the eigenvalue numbering is preserved. Thus, the authors in \cite{tkachuk2014local,gonzalez2020large} directly applied a local counterpart of the global deflation technique (with the same functions $f$ and $g$ as globally) and a uniform rank value $r$. A similar technique was applied patchwise in \cite{voet2024robust} for isogeometric discretizations. \Cref{cor: local_deflation_st2} below provides the bounds in this case.

\begin{corollary}
\label{cor: local_deflation_st2}
If all element matrices are scaled following \Cref{def: deflated_pencil} with the functions $f$ and $g$ defined in \cref{eq: mass_scaling}, then, the global eigenfrequencies $\overline{\omega}_i$ and $\omega_i$ of the scaled and unscaled system, respectively, satisfy
\begin{equation*}
    1 \leq \frac{\omega_i}{\overline{\omega}_i} \leq \max_e \frac{\omega_{m,e}}{\omega_{m-r,e}},
\end{equation*}
where $\omega_{j,e}=\sqrt{\lambda_j(K_e,M_e)}$ are the eigenfrequencies of $(K_e,M_e)$.
\end{corollary}
\begin{proof}
The proof merely requires adapting the choices of $f$ and $g$ in \Cref{cor: local_deflation_st1}.     
\end{proof}

\begin{remark}
If the central difference method is used, \Cref{cor: local_deflation_st2} may be formulated as
\begin{equation*}
    1 \leq \frac{\overline{\Delta t}_c}{\Delta t_c} \leq \max_e \frac{\overline{\Delta t}_{c,e}}{\Delta t_{c,e}}
\end{equation*}
where $\overline{\Delta t}_{c,e}$ and $\Delta t_{c,e}$ are the critical time steps for the locally scaled and unscaled systems, respectively.
\end{remark}

The benefits of local deflation techniques are not always clear. On the one hand, their setup is cheaper than their global counterpart. On the other hand, they might effect the lower frequencies and be less effective at removing the larger ones. Moreover, the Woodbury identity no longer applies to locally deflated systems and solving linear systems with the locally scaled mass matrix is not straightforward. This was the main reason in \cite{voet2024robust} for locally scaling the stiffness matrices instead.

Element modifications have also been suggested for improving the conditioning of system matrices and the stability of immersed methods. In \cite{eisentrager2024eigenvalue}, the element mass (and stiffness) matrices of cut elements are modified by stabilizing near zero eigenvalues. Those techniques simply consist in locally modifying the spectrum of element matrices by adding a stabilization parameter to near zero eigenvalues. Assuming that the $r$ smallest eigenvalues of the mass matrix for a cut element are stabilized with a parameter $\epsilon$ and
\begin{equation*}
    \lambda_1(M_1)+\epsilon \leq \dots \leq \lambda_r(M_e) + \epsilon \leq \lambda_{r+1}(M_e) \leq \dots \leq \lambda_m(M_e),
\end{equation*}
then, if one reasonably assumes that $\min_e \lambda_1(\overline{M}_e)$ is attained for a cut element, the improvement of the condition number of the assembled mass matrix is readily appreciated from \cref{eq: upper_cond_M}:
\begin{equation*}
    \kappa(\overline{M}) \leq \frac{p_{\max}}{\epsilon} \max_e \lambda_m(M_e).
\end{equation*}
 
Unfortunately, local deflation strategies require explicit knowledge of the mode shapes, which is not convenient when the stiffness matrix changes during the course of the simulation. Some alternatives are analyzed next.

\subsubsection{Method of [Olovsson et al., 2005]}
\label{se: olovsson}
In the early 2000s, Olovsson et al. \cite{olovsson2005selective} proposed an ad hoc local mass scaling strategy specifically designed for linear hexahedral elements (and later generalized to other elements). For problems in elasticity, the element lumped mass and scaling matrices are defined as
\begin{equation*}
    M_e=I_3 \otimes \mathrm{M}_e
    \quad \text{and} \quad E_e=I_3 \otimes \mathrm{E}_e,
\end{equation*}
where $\mathrm{M}_e=\frac{m_e}{8}I_8$ with $m_e$ the mass of element $e$ and
\begin{equation*}
    \mathrm{E}_e=\frac{\beta m_e}{56}
    \begin{pmatrix}
        7 & -1 & \cdots & -1 \\
        -1 & 7 & & \vdots\\
        \vdots & & \ddots & \\
        -1 & \cdots & & 7
    \end{pmatrix}
\end{equation*}
where $\beta \geq 0$ is a scaling parameter. In principle $\beta$ may depend on the element but we will not consider such cases here. We immediately notice that $\mathrm{E}_e$ is symmetric positive semidefinite with zero row-sum. Indeed, by the Gershgorin circle theorem \cite[][Theorem 1.1]{varga2011gersgorin}, its (real) eigenvalues lie in the interval $[0, \ \frac{\beta m_e}{4}]$. The positive semi-definiteness carries over to the assembled matrix $E$ and \Cref{lem: equivalence_conditions} guarantees a decrease of the generalized eigenvalues of the scaled matrix pair. Note that $\mathrm{E}_e$ may be expressed as\footnote{Later authors \citep{gavoille2013enrichissement,borrvall2011selective} have defined instead $\mathrm{E}_e = \frac{\beta m_e}{8}(I_8-\mathbf{u}\mathbf{u}^T)$ but we will stick to the original formulation, as proposed in \cite{olovsson2005selective}, where $\beta=1$ doubles the diagonal entries of $M_e$. This small subtlety merely changes a constant.}
\begin{equation}
\label{eq: projector}
    \mathrm{E}_e=\frac{\beta m_e}{7}(I_8-\mathbf{u}\mathbf{u}^T)
\end{equation}
where $\mathbf{u}=\frac{\mathbf{e}}{\|\mathbf{e}\|_2}$ is the normalized vector of all ones. The analogy with local deflation techniques is now evident since \cref{eq: projector} indicates that, up to a multiplicative constant, $\mathrm{E}_e$ is the orthogonal projector into the orthogonal complement of rigid body translations $\Span \{\mathbf{u}\}^\perp$ \cite{gavoille2013enrichissement} and $\mathbf{u} \in \ker(\mathrm{E}_e)$. In the engineering community, this last property is known as the conservation of linear momentum (or translational inertia) \cite{tkachuk2015direct}. Since $M_e$ is a scaled identity matrix (i.e. a scalar matrix), the generalized eigenvectors of $(K_e,M_e)$ are merely eigenvectors of $K_e$. Moreover, since $K_e$ is symmetric, its eigenvectors are orthogonal and based on \cref{eq: CFL_central_difference}, one may easily show that
\begin{equation*}
    \frac{\overline{\Delta t}_{c,e}}{\Delta t_{c,e}} = \sqrt{1+\frac{8}{7}\beta}.
\end{equation*}
In fact, as we will prove below, the right-hand side is also an upper bound on $\frac{\overline{\Delta t}_c}{\Delta t_c}$ and favorably compares with the $\sqrt{1+\beta}$ factor increase reported by Olovsson et al. \cite{olovsson2006iterative}\footnote{The $\sqrt{1+\beta}$ factor increase holds for alternative formulations presented in \citep{gavoille2013enrichissement,borrvall2011selective}. Note the striking similarity with \Cref{cor: local_deflation_st1}.}. Unfortunately, the authors also reported an increase of the condition number of $\overline{M}$ by roughly a factor $1+2\beta$. We will now prove these observations and even refine them. 

\begin{corollary}[Mass scaling of {[Olovsson et al., 2005]}]
\label{cor: bounds_olovsson}
For the mass scaling method of Olovsson et al. \cite{olovsson2005selective}, the following inequalities hold:
\begin{equation*}
    1 \leq \frac{\omega_i}{\overline{\omega}_i} \leq \sqrt{1+\frac{8}{7}\beta}, \qquad \frac{\kappa(\overline{M})}{\kappa(M)} \leq 1+\frac{8}{7}\beta.
\end{equation*}
\end{corollary}
\begin{proof}
We first prove the inequalities on the ratio of eigenfrequencies. Once again, by combining \Cref{th: eig_bounds,th: bounds_generalized_eig},
\begin{equation}
\label{eq: bounds_eigenfrequency_ratio}
     \min_e \lambda_1(\overline{M}_e, M_e) \leq \lambda_1(\overline{M},M) \leq \frac{\lambda_i(K,M)}{\lambda_i(K, \overline{M})} \leq \lambda_n(\overline{M},M) \leq \max_e \lambda_m(\overline{M}_e, M_e).
\end{equation}
Moreover, since the Kronecker product only increases the multiplicity of the eigenvalues of $(\overline{\mathrm{M}}_e, \mathrm{M}_e)$,
\begin{equation*}
    \Lambda(\overline{M}_e, M_e)=\Lambda(\overline{\mathrm{M}}_e, \mathrm{M}_e).
\end{equation*}
Additionally, due to the simple structure of $\mathrm{M}_e$, the eigenvalues of $(\overline{\mathrm{M}}_e, \mathrm{M}_e)$ are known explicitly. Denoting $\gamma_e=\frac{m_e}{56}$, we obtain $\mathrm{M}_e=7\gamma_e I_8$ and recalling \cref{eq: projector}, $\mathrm{E}_e=\beta\gamma_e(8I_8-\mathbf{e}\mathbf{e}^T)$, which is the sum of a scaled identity and rank-1 matrix. Consequently, the scaled mass matrix is given by
\begin{equation*}
    \overline{\mathrm{M}}_e=\mathrm{M}_e+\mathrm{E}_e=\gamma_e((7+8\beta)I_8-\beta \mathbf{e}\mathbf{e}^T).
\end{equation*}
Thus,
\begin{equation*}
    \Lambda(\overline{\mathrm{M}}_e, \mathrm{M}_e)=\Lambda(\gamma_e((7+8\beta)I_8-\beta \mathbf{e}\mathbf{e}^T), 7\gamma_e I_8)=\Lambda((1+\frac{8}{7}\beta) I_8-\frac{\beta}{7}\mathbf{e}\mathbf{e}^T).
\end{equation*}
The eigenvalues of this last matrix are known explicitly. Indeed, since $\lambda_8(\mathbf{e}\mathbf{e}^T)=\|\mathbf{e}\|_2^2=8$ and $\lambda_k(\mathbf{e}\mathbf{e}^T)=0$ for $k=1,\dots,7$, we deduce that
\begin{equation}
\label{eq: element_eig_olovsson}
\lambda_k(\overline{\mathrm{M}}_e, \mathrm{M}_e)=
\begin{cases}
1 & \text{ for } k=1, \\
1+\frac{8}{7}\beta & \text{ for } k=2,\dots,8.
\end{cases}
\end{equation}
The result for the eigenfrequency ratio then immediately follows by taking the square root in \eqref{eq: bounds_eigenfrequency_ratio}.

For the upper bound on the condition number, we invoke \Cref{cor: conditioning_bound},
\begin{equation}
\label{eq: bound_conditioning}
    \frac{\kappa(\overline{M})}{\kappa(M)} \leq \kappa(\overline{M},M)=\frac{\lambda_n(\overline{M},M)}{\lambda_1(\overline{M},M)} \leq \frac{\max_e \lambda_m(\overline{M}_e, M_e)}{\min_e \lambda_1(\overline{M}_e, M_e)}=\frac{\max_e \lambda_8(\overline{\mathrm{M}}_e, \mathrm{M}_e)}{\min_e \lambda_1(\overline{\mathrm{M}}_e, \mathrm{M}_e)}.
\end{equation}
where the second inequality follows from \Cref{th: bounds_generalized_eig} and the last equality follows from eigenvalue multiplicity. Finally, recalling \eqref{eq: element_eig_olovsson}, the result follows.
\end{proof}

\begin{remark}
Interestingly, none of the bounds depend on $m_e$, which bears the element dependency. A more explicit upper bound on the condition number of $\overline{M}$ can be derived from \Cref{lem: conditioning_FEM_matrices} and \cref{eq: upper_cond_M}:
\begin{equation}
\label{eq: upper_cond_Ms}
    \kappa(\overline{M}) \leq p_{\max}(1+\frac{8}{7}\beta) \frac{\max_e m_e}{\min_e m_e}.
\end{equation}
One can easily see that for a sufficiently refined uniform mesh of linear hexahedral elements with uniform density, the condition number of $M$ is $p_{\max}$. Indeed, all element lumped mass matrices are given by $M_e=\alpha I$, where $\alpha$ is a constant independent of the element and is contributed exactly once for a corner degree of freedom while it is contributed $p_{\max}$ times for an interior degree of freedom. In such cases, the bounds of \Cref{cor: bounds_olovsson} and \cref{eq: upper_cond_Ms} coincide. Moreover, for the central difference method, \Cref{cor: bounds_olovsson} also immediately yields a bound on the ratio of critical time steps:
\begin{equation*}
    1 \leq \frac{\overline{\Delta t}_c}{\Delta t_c} = \frac{\omega_n}{\overline{\omega}_n} \leq \sqrt{1+\frac{8}{7}\beta}.
\end{equation*}
In particular, we notice that the upper bound is very close to the estimate of $\sqrt{1+\beta}$ given in \cite{olovsson2006iterative}. Furthermore, the proof strategy of \Cref{cor: bounds_olovsson} can be easily adapted to other elements.
\end{remark}
 
\subsubsection{Method of [Hoffmann et al., 2023]}
\label{se: hoffmann}
A rather straightforward generalization of the method of Olovsson et al. consists in defining
\begin{equation}
\label{eq: projector_gen}
    \mathrm{E}_e=\frac{\beta m_e}{7}(I_8-UU^T)
\end{equation}
where $U \in \mathbb{R}^{8 \times r}$ is a matrix whose columns form an orthonormal basis for the low-frequency mode shapes, including rigid body translations and rotations. In engineering terms, this mass scaling strategy preserves both linear and angular momentum (or translational and rotational inertia). By including more of the low-frequency content locally, one may hope for greater accuracy globally. Recently, Hoffmann et al. \cite{hoffmann2023finite} suggested an improvement of the method of Olovsson et al. somewhat mimicking this idea and claiming similar increases on the critical time step while achieving significantly better accuracy for the smallest eigenvalues. The element scaling matrix is defined as $E_e= I_3 \otimes \mathrm{E}_e$, where
\begin{equation*}
    \mathrm{E}_e = \frac{\beta \widetilde{\gamma}_e}{4}(A \otimes G)
\end{equation*}
with
\begin{equation*}
    A =
    \begin{pmatrix}
        1 & -1 \\
        -1 & 1
    \end{pmatrix},
    \quad \text{and} \quad
    G =
    \begin{pmatrix}
        4 & 2 & 1 & 2 \\
        2 & 4 & 2 & 1 \\
        1 & 2 & 4 & 2 \\
        2 & 1 & 2 & 4
    \end{pmatrix}.
\end{equation*}
The factor $\widetilde{\gamma}_e$ depends on the element geometry and material but is not explicitly defined by the authors in \cite{hoffmann2023finite}. However, since their method builds on the work of Olovsson et al., we assume that $\tilde{\gamma}_e=7\gamma_e=\frac{m_e}{8}$ such that $\beta=1$ doubles the diagonal entries of $M_e$, similarly to the method of Olovsson et al.

\begin{corollary}[Mass scaling of {[Hoffmann et al., 2023]}]
\label{cor: bounds_hoffmann}
For the mass scaling method of Hoffmann et al. \cite{hoffmann2023finite}, the following inequalities hold:
\begin{equation*}
    1 \leq \frac{\omega_i}{\overline{\omega}_i} \leq \sqrt{1+\frac{9}{2}\beta}, \qquad \frac{\kappa(\overline{M})}{\kappa(M)} \leq 1+\frac{9}{2}\beta.
\end{equation*}
\end{corollary}
\begin{proof}
The proof arguments follow exactly the same lines as in \Cref{cor: bounds_olovsson}. Only the definition of the scaling matrix changes and so do the eigenvalues of $(\overline{\mathrm{M}}_e, \mathrm{M}_e)$. Yet, computing them is again a simple exercise:
\begin{equation*}
    \Lambda(\overline{\mathrm{M}}_e, \mathrm{M}_e) = \Lambda(7\gamma_e ( I_8 + \frac{\beta}{4}(A \otimes G)), 7\gamma_e I_8) = 1+\frac{\beta}{4}\Lambda(A)\Lambda(G),
\end{equation*}
where the multiplication of two sets is the set containing all elementwise multiplications of any two elements from the sets. The spectrum of $A$ and $G$ (with multiplicities) is
\begin{equation*}
    \Lambda(A) = \{0,2\}, \quad \text{and} \quad \Lambda(G) = \{1,3,3,9\}.
\end{equation*}
Consequently,
\begin{equation}
\label{eq: element_eig_hoffmann}
\lambda_k(\overline{\mathrm{M}}_e, \mathrm{M}_e)=
\begin{cases}
1 & \text{ for } k=1,2,3,4, \\
1+\frac{\beta}{2} & \text{ for } k=5, \\
1+\frac{3\beta}{2} & \text{ for } k=6,7, \\
1+\frac{9\beta}{2} & \text{ for } k=8.
\end{cases}
\end{equation}
The results then follow from direct substitution in \eqref{eq: bounds_eigenfrequency_ratio} and \eqref{eq: bound_conditioning}.
\end{proof}
In particular, \Cref{cor: bounds_olovsson,cor: bounds_hoffmann} suggest that increasing the critical time step and retaining a moderate condition number for the scaled mass matrix are conflicting objectives. This has already been observed numerically by several authors and will be confirmed in the next section.

\section{Numerical validation}
\label{se: numerical_experiments}
This section provides a few numerical experiments supporting our theoretical findings. Our experiments specifically focus on local methods, when the transformed eigenvalues of the global system are unknown. We refer to the original articles cited herein for further experiments assessing the accuracy, especially for transient problems. The values reported for the critical time step are the ones for the central difference method \cref{eq: CFL_central_difference}, ubiquitous in explicit dynamics. All experiments were done with an in-house finite element code while the meshes were generated with GMSH \cite{geuzaine2009gmsh}.

\subsection{Ad hoc local SMS}
The global behavior of the methods of Olovsson et al. \cite{olovsson2005selective} and Hoffmann et al. \cite{hoffmann2023finite} is closely tied to the element constructions. For illustrating it, we first consider a single thin linear hexahedral element of size $1 \times 1 \times 10^{-3}$, whose Young modulus, Poisson ratio and density are $E=207$ GPa, $\nu=0.3$ and $\rho=7800$ kg/m\textsuperscript{3}, respectively. The element (lumped) mass matrix $M_e$ is a scalar matrix such that the eigenvectors of $(K_e,M_e)$ are simply eigenvectors of the stiffness matrix. For the method of Olovsson et al., all eigenvectors $\mathbf{u}_k$ for $k=7,8,\dots,m$ (i.e. $\mathbf{u}_k \in \mathbb{R}^m \setminus \ker(K_e)$) are also generalized eigenvectors of $(\overline{M}_e, M_e)$ since they are in the kernel of $I_3 \otimes \mathbf{e}\mathbf{e}^T$. The associated eigenvalues are $1+\frac{8}{7}\beta$ according to \cref{eq: element_eig_olovsson} and it follows from \Cref{lem: common_eigenvector} that the eigenvalues of the scaled matrix pair are 
\begin{equation*}
    \lambda_{i_k}(K_e,\overline{M}_e) = \frac{\lambda_k(K_e,M_e)}{Q_e(\mathbf{u}_k)} = \frac{\lambda_k(K_e,M_e)}{1+\frac{8}{7}\beta} \quad \text{where} \quad Q_e(\mathbf{x})=\frac{\mathbf{x}^T \overline{M}_e \mathbf{x}}{\mathbf{x}^T M_e \mathbf{x}}
\end{equation*}
is the Rayleigh quotient. In this case, the perturbation scales down all eigenvalues uniformly and therefore preserves the eigenvalue ordering (i.e. $i_k=k$) for $k \geq 7$. These results are shown graphically in \Cref{fig: 3D_Elasticity_plate_FEM_Olovsson_eig_element} for $\beta=1$ and confirm our theoretical findings. The results are far more interesting when considering the method of Hoffmann et al. Also in this case it turns out all eigenvectors of $(K_e,M_e)$ are eigenvectors of $(\overline{M}_e,M_e)$. The deep understanding of the deformation modes from engineering practice certainly guided the construction of the scaling matrix. These modes bear different names and are carefully listed and grouped in \citep[][Figure 3]{cocchetti2013selective}. Applying \Cref{lem: common_eigenvector} once again, we deduce that
\begin{equation*}
    \lambda_{i_k}(K_e,\overline{M}_e) = \frac{\lambda_k(K_e,M_e)}{Q_e(\mathbf{u}_k)}.
\end{equation*}
where the overline now obviously refers to scaled quantities for the method of Hoffmann et al. In this case, $Q_e(\mathbf{u}_k)$ may take any value among those listed in \cref{eq: element_eig_hoffmann}. Therefore, the scaling may not preserve the eigenvalue ordering. The values of $Q_e(\mathbf{u}_k)$ for $k=7,\dots,m$ are shown in \Cref{fig: 3D_Elasticity_plate_FEM_Hoffmann_eig_element} for $\beta=1$ alongside the eigenvalues of the scaled and original matrix pairs. It appears that $\lambda_{i_{20}}(K_e,\overline{M}_e)=\lambda_{i_{21}}(K_e,\overline{M}_e) < \lambda_{i_{19}}(K_e,\overline{M}_e)$ and confirms that the scaling does not preserve the eigenvalue ordering. Generally speaking, comparing the $k$th largest eigenvalues for the original and scaled matrix pairs is only relevant for ``small'' perturbations. In other cases, one should better understand how the original eigenvalues are transformed.

\begin{figure}[H]
    \centering
    \includegraphics[width=0.8\linewidth]{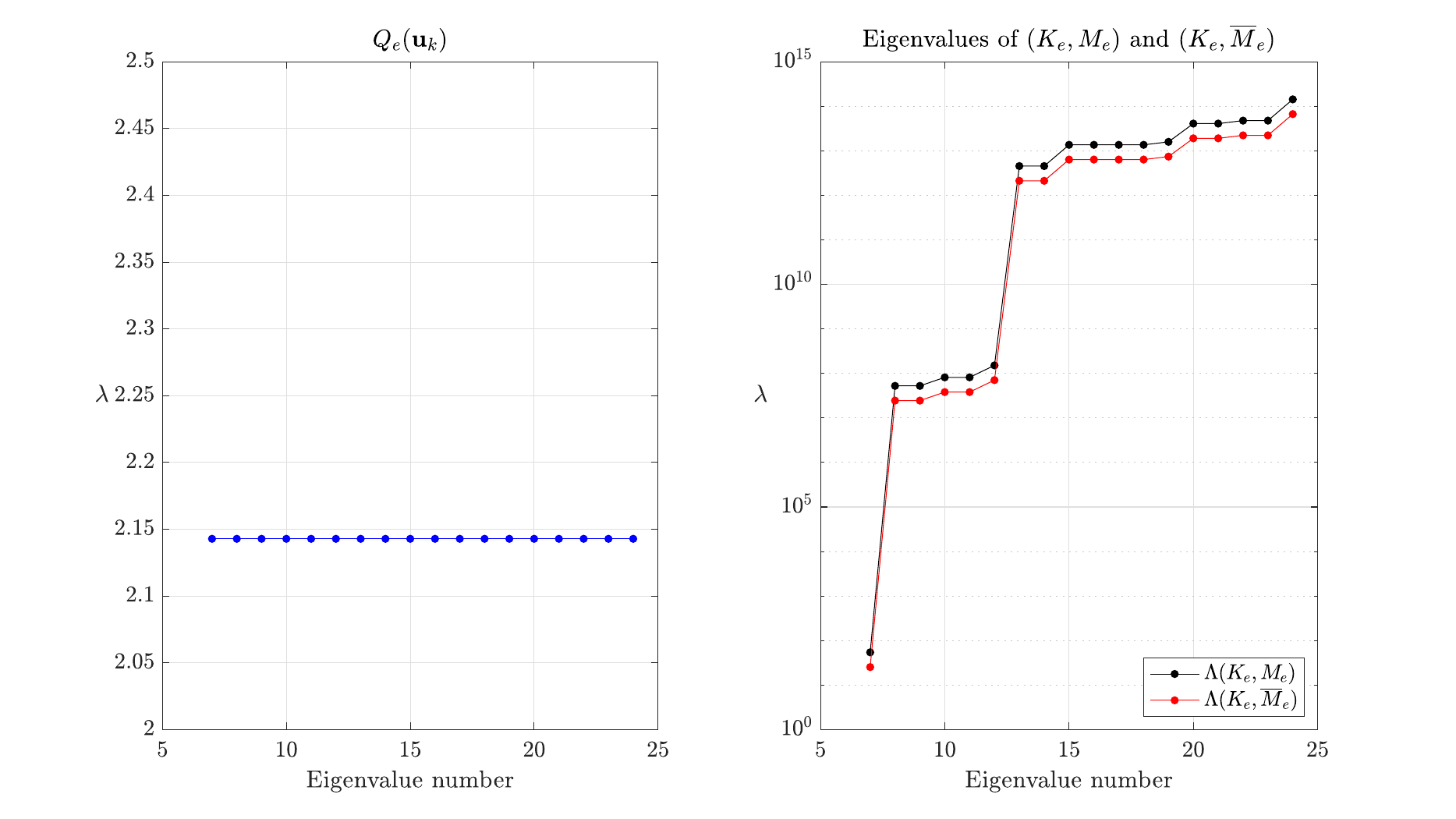}
    \caption{Values of $Q_e(\mathbf{u}_k)$ for $k=7,\dots,m$ and eigenvalues of $(K_e,M_e)$ and $(K_e,\overline{M}_e)$ for $\beta=1$ and the method of Olovsson et al.}
    \label{fig: 3D_Elasticity_plate_FEM_Olovsson_eig_element}
\end{figure}

\begin{figure}[H]
    \centering
    \includegraphics[width=0.8\linewidth]{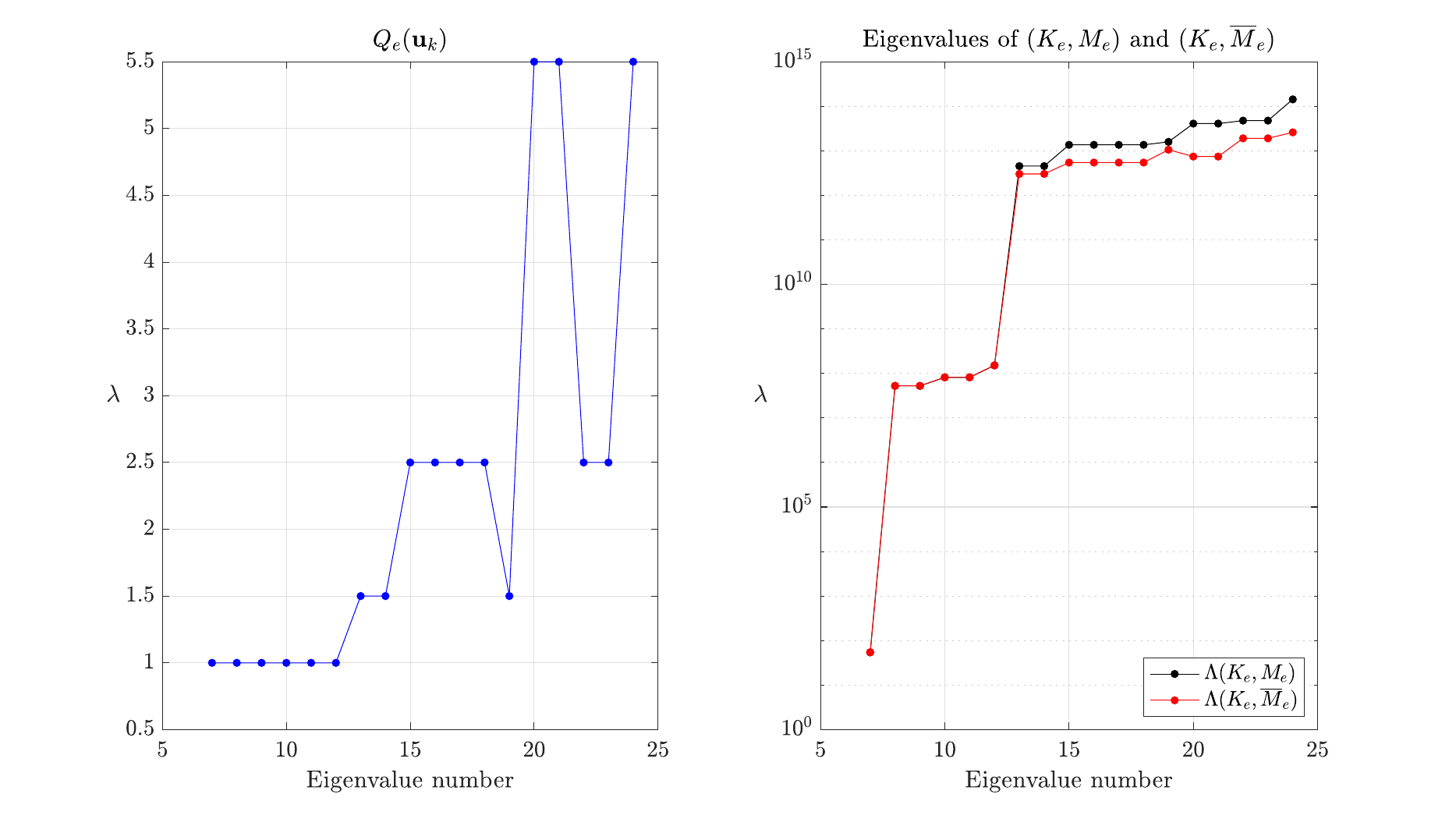}
    \caption{Values of $Q_e(\mathbf{u}_k)$ for $k=7,\dots,m$ and eigenvalues of $(K_e,M_e)$ and $(K_e,\overline{M}_e)$ for $\beta=1$ and the method of Hoffmann et al.}
    \label{fig: 3D_Elasticity_plate_FEM_Hoffmann_eig_element}
\end{figure}

The insights drawn locally will prove useful for illustrating some global properties. As benchmark, we consider a thin metal strip of length $\times$ width $\times$ thickness dimensions $200 \times 20 \times 2$ mm, whose Young modulus, Poisson ratio and density are again $E=207$ GPa, $\nu=0.3$ and $\rho=7800$ kg/m\textsuperscript{3}, respectively. The plate is discretized with $40 \times 5 \times 4$ nodes in each respective dimension, producing $348$ linear hexahedral elements. The resulting mesh is shown in \Cref{fig: plate_mesh_n40_5_4}. In this example, we test the tightness of the bounds in \Cref{cor: bounds_olovsson,cor: bounds_hoffmann} for the methods of Olovsson et al. \cite{olovsson2005selective} and Hoffmann et al. \cite{hoffmann2023finite}, respectively. The increase of the critical time step is displayed in \Cref{fig: 3D_Elasticity_plate_FEM_delta_t_crit_hex8} along with its upper bound for different values of the scaling parameter $\beta$ ranging from $1$ to $500$. For the method of Olovsson et al., the upper bound is very sharp over the entire range of values tested, whereas for the method of Hoffmann et al., the upper bound is satisfactory only for moderate values of $\beta$. For larger values, the ratio flattens out and the critical time step most likely becomes constrained by other deformation modes, not targeted by the method of Hoffmann et al. Nevertheless, both methods may achieve a similar increase of the critical time step, albeit for different values of $\beta$.

\begin{figure}[H]
    \centering
    \includegraphics[width=0.8\linewidth]{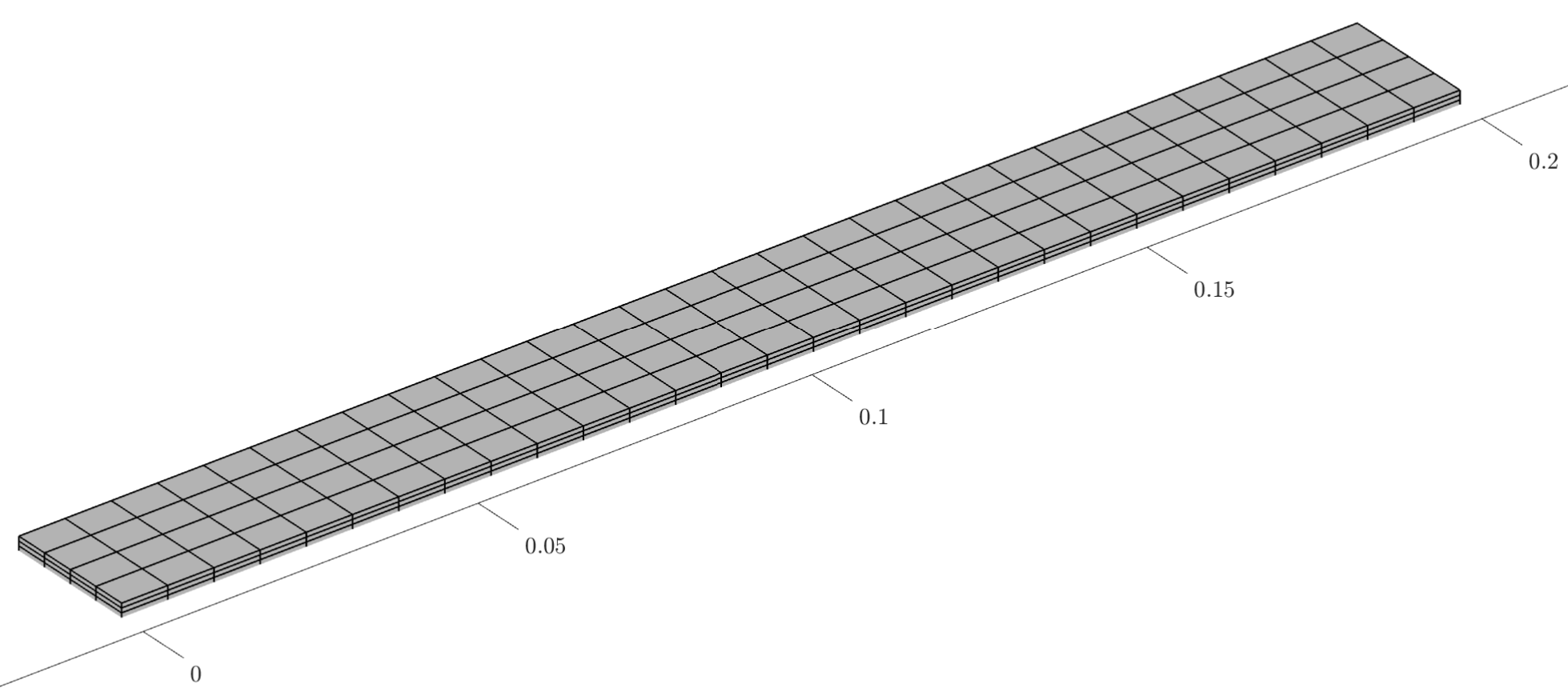}
    \caption{Hexahedral finite element mesh}
    \label{fig: plate_mesh_n40_5_4}
\end{figure}

\begin{figure}[H]
    \centering
    \includegraphics[width=0.5\linewidth]{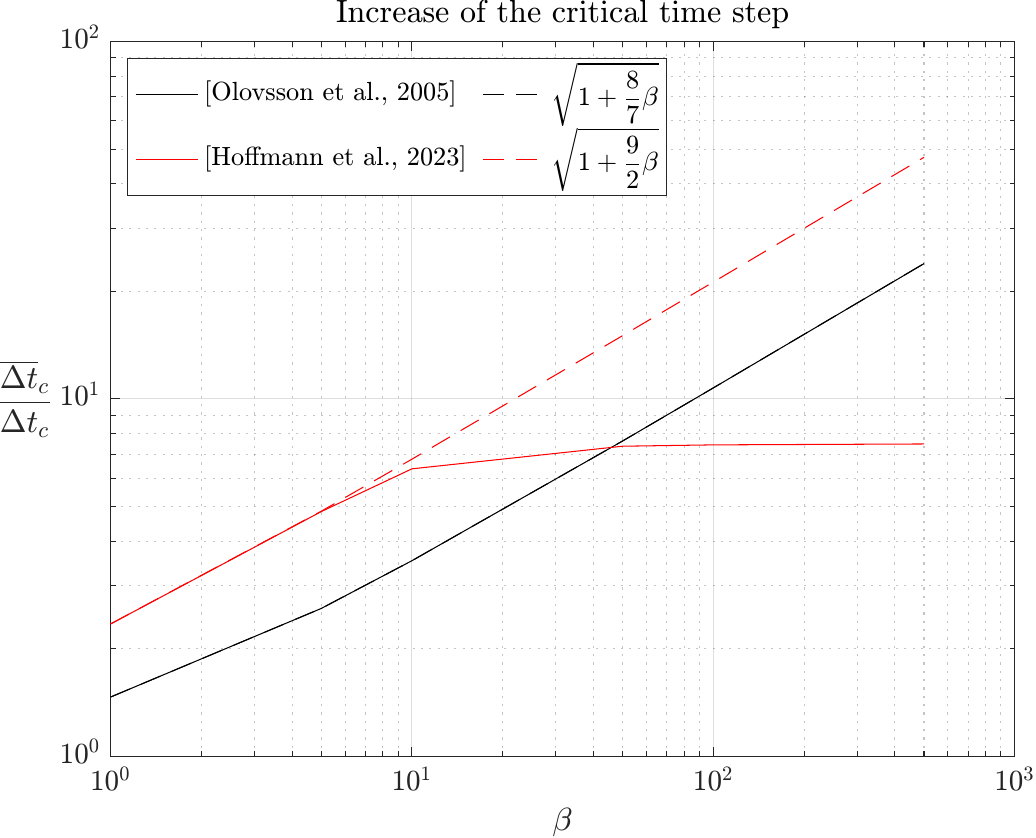}
    \caption{Increase of the critical time step for the methods of Olovsson et al. and Hoffmann et al.}
    \label{fig: 3D_Elasticity_plate_FEM_delta_t_crit_hex8}
\end{figure}

Following standard practice, \Cref{fig: 3D_Elasticity_plate_FEM_eig_ratio_hex8} show the ratio of scaled over unscaled frequencies for selected values of $\beta$. Note that the upper bounds of \Cref{cor: bounds_olovsson,cor: bounds_hoffmann} (valid over the entire spectrum) become lower bounds in this case. Admittedly, except preserving perhaps the first couple of eigenfrequencies, the method of Olovsson et al. seems quite inaccurate and \Cref{fig: 3D_Elasticity_plate_FEM_Olovsson_eig_ratio_hex8} fuels existing concerns raised in \citep{tkachuk2014local,hoffmann2023finite}. On the contrary, the method of Hoffmann et al. seems vastly superior, achieving a similar increase of the critical time step while yielding greater accuracy of the lower frequencies. Interestingly, the overall trend reported in \Cref{fig: 3D_Elasticity_plate_FEM_eig_ratio_hex8}. is reminiscent of the local construction and the targeted modes.

\begin{figure}[H]
    \centering
    \begin{subfigure}[t]{0.48\textwidth}
    \centering
    \includegraphics[width=\textwidth]{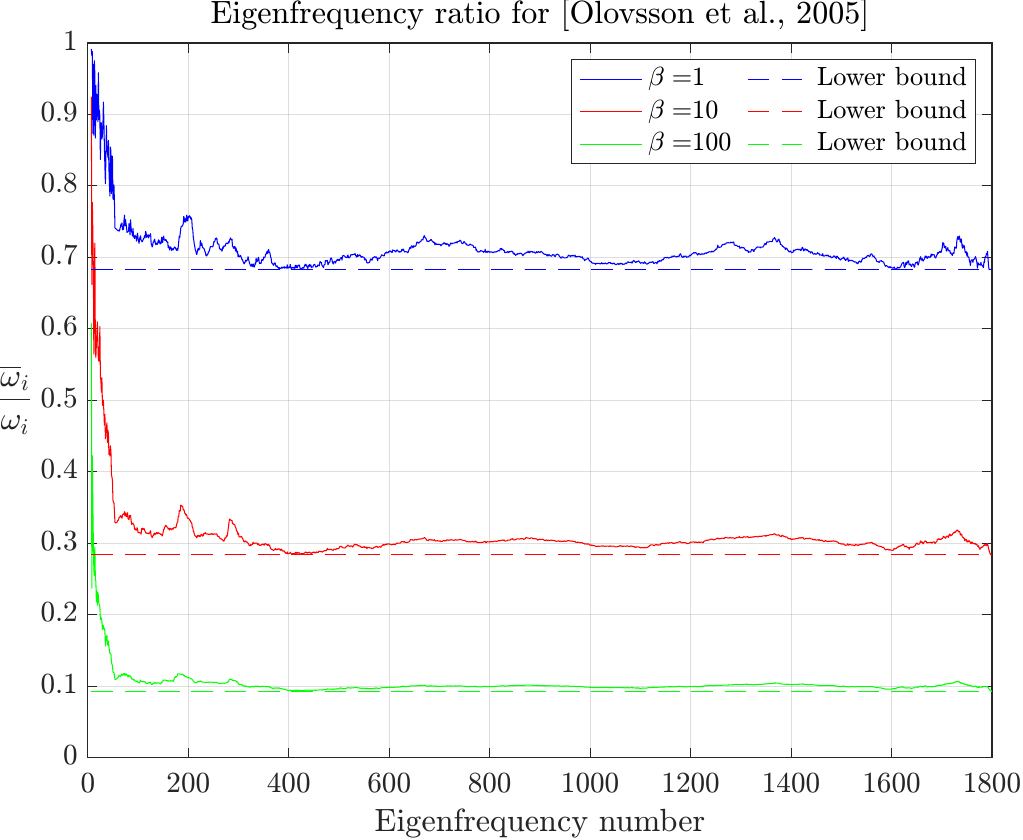}
    \caption{Method of Olovsson et al.}
    \label{fig: 3D_Elasticity_plate_FEM_Olovsson_eig_ratio_hex8}
    \end{subfigure}
    \hfill
    \begin{subfigure}[t]{0.48\textwidth}
    \centering
    \includegraphics[width=\textwidth]{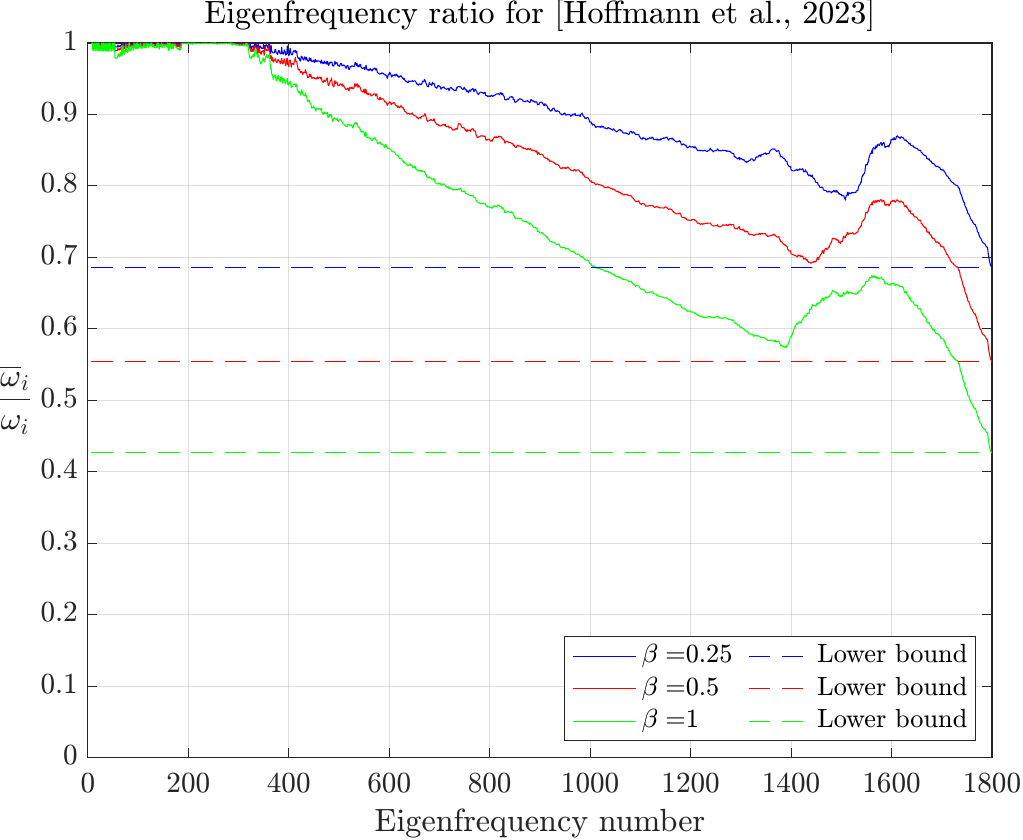}
    \caption{Method of Hoffmann et al.}
    \label{fig: 3D_Elasticity_plate_FEM_Hoffmann_eig_ratio_hex8}
    \end{subfigure}
    \hfill
    \caption{Eigenfrequency ratio for different values of the scaling parameter $\beta$}
    \label{fig: 3D_Elasticity_plate_FEM_eig_ratio_hex8}
\end{figure}

If linear systems with the scaled mass matrix are solved iteratively, as advocated in \cite{olovsson2006iterative}, its condition number is often monitored. The upper bounds in \Cref{cor: bounds_olovsson,cor: bounds_hoffmann} suggest an increase in the number of iterations of iterative solvers (e.g. Conjugate Gradients) as the scaling parameter grows larger and may potentially offset the saving in the number of time steps. Fortunately, as shown in \Cref{fig: 3D_Elasticity_plate_FEM_cond_hex8}, the upper bounds seem quite pessimistic in this case. Although the growth is indeed linear, the rate is smaller than predicted. The numerically observed rates are shown in dotted line. The reason for the discrepancy between the theoretical rates and the observed ones is quite subtle and is due to a step in the proof of \Cref{lem: conditioning_FEM_matrices}. Namely, the bounds $1 \leq \|\mathsf{L}\mathbf{x}\|_2^2 \leq p_{\max}$ may be loose for specific choices of $\mathbf{x}$. For the method of Olovsson et al., one may derive a refined bound in the asymptotic regime as $\beta \to \infty$. For this purpose, we recall that for a lumped mass matrix $M$ with uniform elements and material density, $\kappa(M)=p_{\max}$, where $p_{\max} \in \mathbb{N}^*$ is the maximum number of elements to which a node is connected. In our example, $p_{\max}=8$ and is attained for interior nodes. Furthermore, from the proof of \Cref{lem: conditioning_FEM_matrices}, the following bounds hold:
\begin{equation*}
    \min_e \lambda_1(\overline{M}_e) \|\mathsf{L}\mathbf{x}\|_2^2 \leq \mathbf{x}^T \overline{M}\mathbf{x} \leq \max_e \lambda_m(\overline{M}_e) \|\mathsf{L}\mathbf{x}\|_2^2,
\end{equation*}
where 
\begin{equation*}
    \overline{M}=\sum_{e=1}^N L_e^T\overline{M}_eL_e, \quad \overline{M}_e = I_3 \otimes \overline{\mathrm{M}}_e
\end{equation*}
and\begin{equation*}
    \overline{\mathrm{M}}_e=\mathrm{M}_e+\mathrm{E}_e=\gamma_e((7+8\beta)I_8-\beta \mathbf{e}\mathbf{e}^T) = \gamma_e(7 I_8 + \beta (8I_8-\mathbf{e}\mathbf{e}^T)).
\end{equation*}
As previously noted, $\mathrm{E}_e = \gamma_e\beta (8I_8-\mathbf{e}\mathbf{e}^T) \succeq 0$ and its norm increases with $\beta$. Therefore, as $\beta \to \infty$, the eigenvector $\mathbf{x}_1$ associated to the smallest eigenvalue $\lambda_1(\overline{M})$ will try to annihilate this term. Since $\mathbf{e} \in \ker(E_e)$ and $L_e\mathbf{x}$ simply extracts entries from $\mathbf{x}$, we expect $\mathbf{x}_1$ to converge to the (normalized) vector of all ones; i.e. $\mathbf{x}_1 =\frac{1}{\sqrt{n}}\mathbf{1}$. In this case,
\begin{equation*}
\|\mathsf{L}\mathbf{x}_1\|_2^2 = \frac{1}{n} \sum_{i=1}^n p_i = \frac{1}{n} \sum_{e=1}^N \sum_{j=1}^m 1 = \frac{mN}{n},
\end{equation*}
reduces to the average connectivity, where $N$ is the number of elements, $m$ is the number of degrees of freedom per element and $n$ is the total number of degrees of freedom. Finally, putting all the pieces back together,
\begin{equation*}
    \frac{\kappa(\overline{M})}{\kappa(M)} \leq \frac{\max_e \lambda_m(\overline{M}_e) }{\min_e \lambda_1(\overline{M}_e) \|\mathsf{L}\mathbf{x}_1\|_2^2} \lesssim  \frac{n}{mN} \frac{7+8\beta}{7} \frac{\max_e \gamma_e}{\min_e \gamma_e} = \frac{n}{mN}\left(1 + \frac{8}{7}\beta \right).
\end{equation*}
Substituting the discretization parameters, we obtain $\frac{8n}{7mN} \approx 0.2463$, which is very close to the observed rate of $0.25$. Our findings also consistently match the observations of Olovsson et al. \cite{olovsson2006iterative}, where the (non-normalized) condition number $\kappa(\overline{M})$ increased by roughly $1+2 \beta$. However, this rate is only expected to hold asymptotically and under the previously stated assumptions.

\begin{figure}[H]
    \centering
    \includegraphics[width=0.5\linewidth]{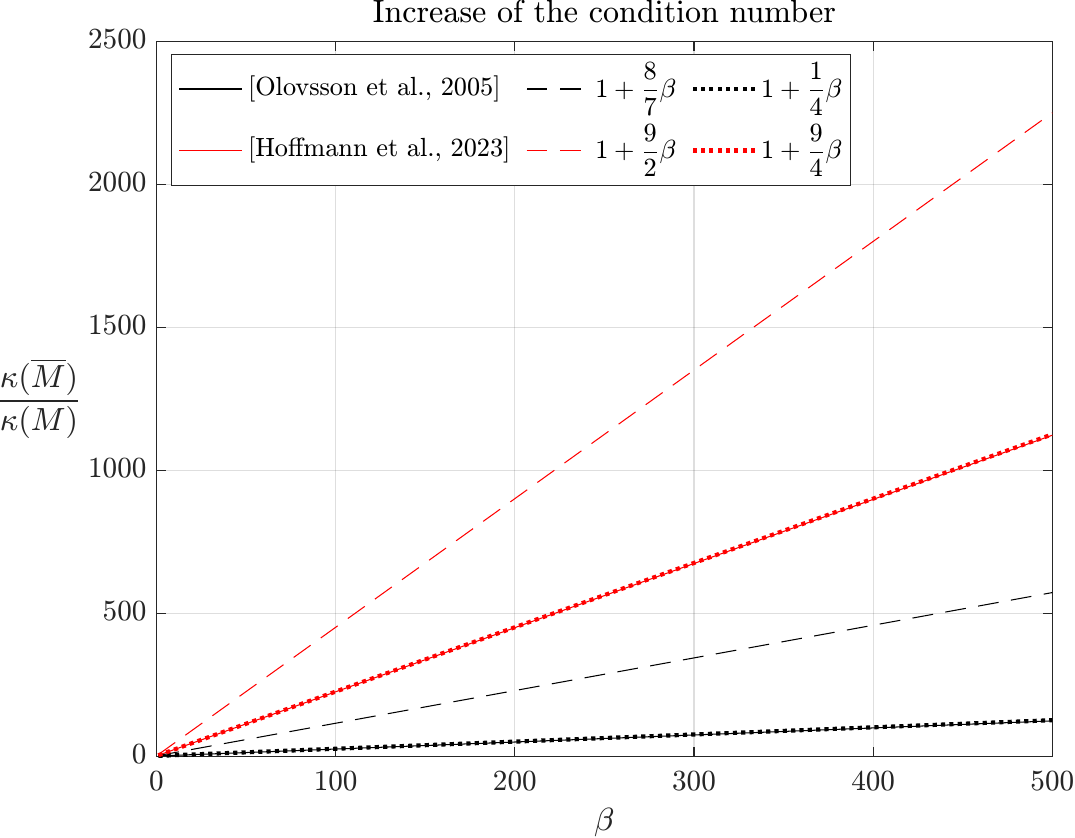}
    \caption{Increase of the condition number for the methods of Olovsson et al. and Hoffmann et al. for a uniform discretization with uniform material properties.}
    \label{fig: 3D_Elasticity_plate_FEM_cond_hex8}
\end{figure}

Nevertheless, the theoretically sound upper bound of $1+\frac{8}{7}\beta$ was much sharper on non-uniform meshes, suggesting that the behavior of the condition number is strongly problem dependent. Performing a similar analysis for the method of Hoffmann et al. is far more involved and does not seem obvious at this stage.

\subsection{Local deflation}
We consider the same benchmark as in the previous section but employ the local deflation techniques presented in \Cref{se: local_deflation}. These techniques scale the element matrix pairs $(K_e,M_e)$ by simply applying \Cref{def: deflated_pencil} locally; i.e.
\begin{align*}
    \overline{K}_e &= K_e+V_ef(D_{e,2})V_e^T, \\
    \overline{M}_e &= M_e+V_eg(D_{e,2})V_e^T,
\end{align*}
where $V_e=M_e U_{e,2}$, $D_{e,2}$ is the diagonal matrix containing the $r$ largest eigenvalues of $(K_e,M_e)$ and $U_{e,2}$ contains the associated eigenvectors along its columns. Two specific choices of functions $f$ and $g$ where considered in \Cref{se: local_deflation}:
\begin{equation*}
    \text{Strategy 1: }
    \begin{cases}
        f(\lambda)=0, & \\
        g(\lambda)=\alpha,
    \end{cases}
    \qquad \text{Strategy 2: }
    \begin{cases}
        f(\lambda)=0, & \\
        g(\lambda)=\frac{\lambda}{\lambda_{m-r}(K_e,M_e)}-1,
    \end{cases}
\end{equation*}
where $\alpha>0$ is a scaling parameter. In both cases the stiffness is unaltered ($\overline{K}_e=K_e$) and the eigenvalues of the scaled matrix pairs $(\overline{K}_e, \overline{M}_e)$ are (see \Cref{th: low_rank_pert_AB})
\begin{align*}
    &\text{Strategy 1: }
    \lambda_{i_k}(\overline{K}_e, \overline{M}_e)=
    \begin{cases}
    \lambda_k(K_e,M_e) & \text{ for } k=1,\dots,m-r, \\
    \frac{\lambda_k(K_e,M_e)}{1+\alpha} & \text{ for } k=m-r+1,\dots,m.
    \end{cases} \\
    &\text{Strategy 2: }
    \lambda_k(\overline{K}_e, \overline{M}_e)=
    \begin{cases}
    \lambda_k(K_e,M_e) & \text{ for } k=1,\dots,m-r, \\
    \lambda_{m-r}(K_e,M_e) & \text{ for } k=m-r+1,\dots,m.
    \end{cases}
\end{align*}
Once again, local properties provide valuable insight on the global behavior and we consider for the time being a single thin linear hexahedral element of size $1 \times 1 \times 10^{-3}$ with the same material properties as the previous example. By construction, the matrix pairs $(K_e,M_e)$ and $(\overline{M}_e,M_e)$ share the same eigenvectors and \Cref{lem: common_eigenvector} applies. Note, however, that the first strategy may not preserve the eigenvalue ordering ($i_k \neq k)$. This happens whenever $\alpha$ is large relative to the eigenvalue gap:
\begin{equation}
\label{eq: eig_gap}
    \alpha > \frac{\lambda_{m-r+1}(K_e,M_e)}{\lambda_{m-r}(K_e,M_e)}-1.
\end{equation}
Consequently, for the first strategy, the largest eigenvalue of the scaled matrix pair is
\begin{equation*}
    \lambda_m(\overline{K}_e, \overline{M}_e) = \max \left\{\lambda_{m-r}(K_e,M_e), \ \frac{\lambda_m(K_e,M_e)}{1+\alpha}\right\}.
\end{equation*}
Thus, as $\alpha$ gets increasingly large, the largest eigenvalue is given by $\lambda_{m-r}(K_e,M_e)$, which no longer depends on $\alpha$ but only on the deflation rank and therefore places a threshold on the maximum increase of the critical time step. \Cref{fig: 3D_Elasticity_plate_FEM_local_defl_st1_eig_element_r12_alpha1,fig: 3D_Elasticity_plate_FEM_local_defl_st1_eig_element_r7_alpha1e2} show two different cases where the gap assumption \cref{eq: eig_gap} is satisfied and violated, respectively. For the second example, one also notices that $\lambda_m(\overline{K}_e, \overline{M}_e)=\lambda_{m-r}(K_e,M_e)$.

\begin{figure}[H]
    \centering
    \includegraphics[width=0.8\linewidth]{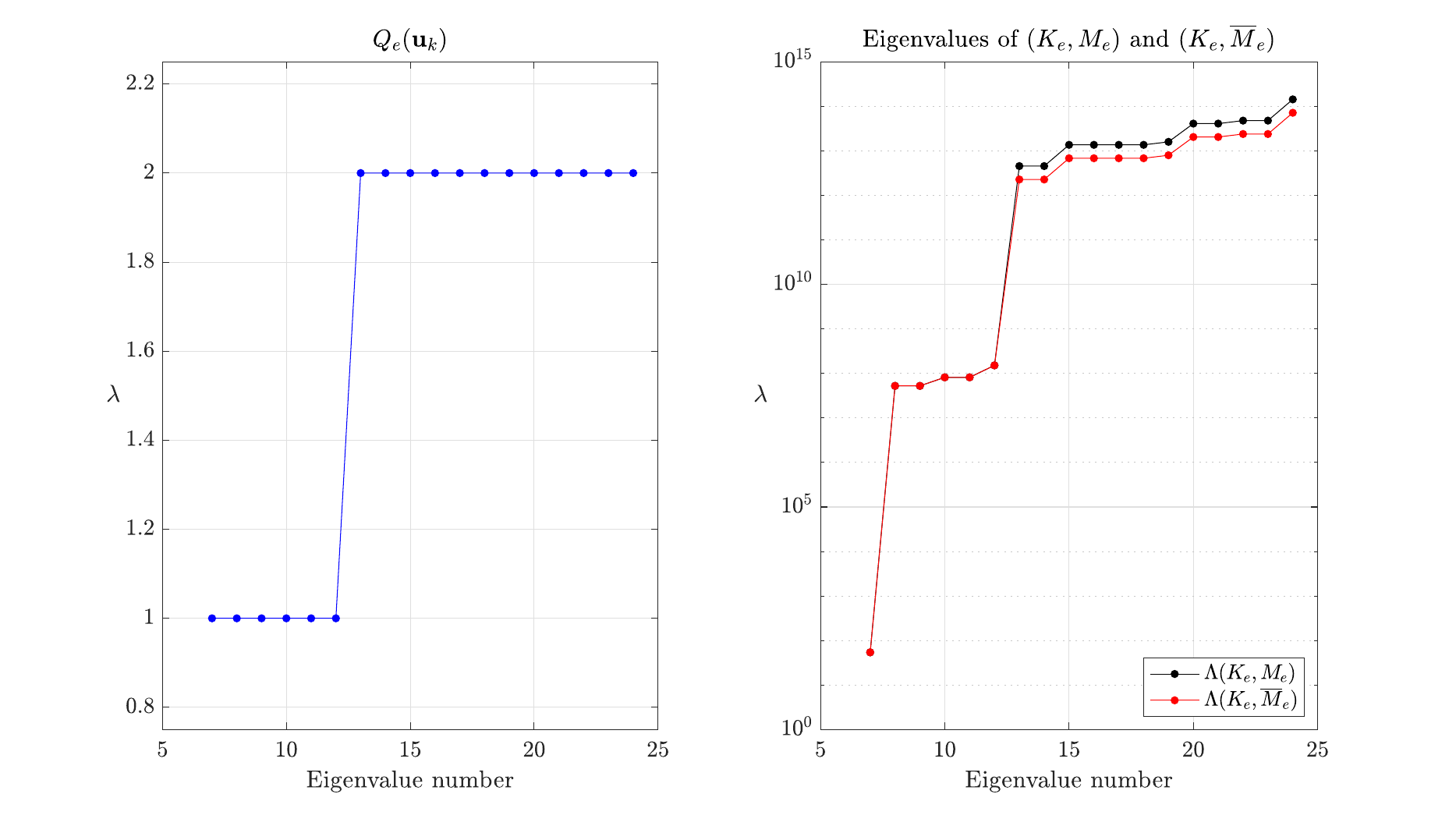}
    \caption{Values of $Q_e(\mathbf{u}_k)$ for $k=7,\dots,m$ and eigenvalues of $(K_e,M_e)$ and $(K_e,\overline{M}_e)$ for the first strategy with $r=12$ and $\alpha=1$}
    \label{fig: 3D_Elasticity_plate_FEM_local_defl_st1_eig_element_r12_alpha1}
\end{figure}

\begin{figure}[H]
    \centering
    \includegraphics[width=0.8\linewidth]{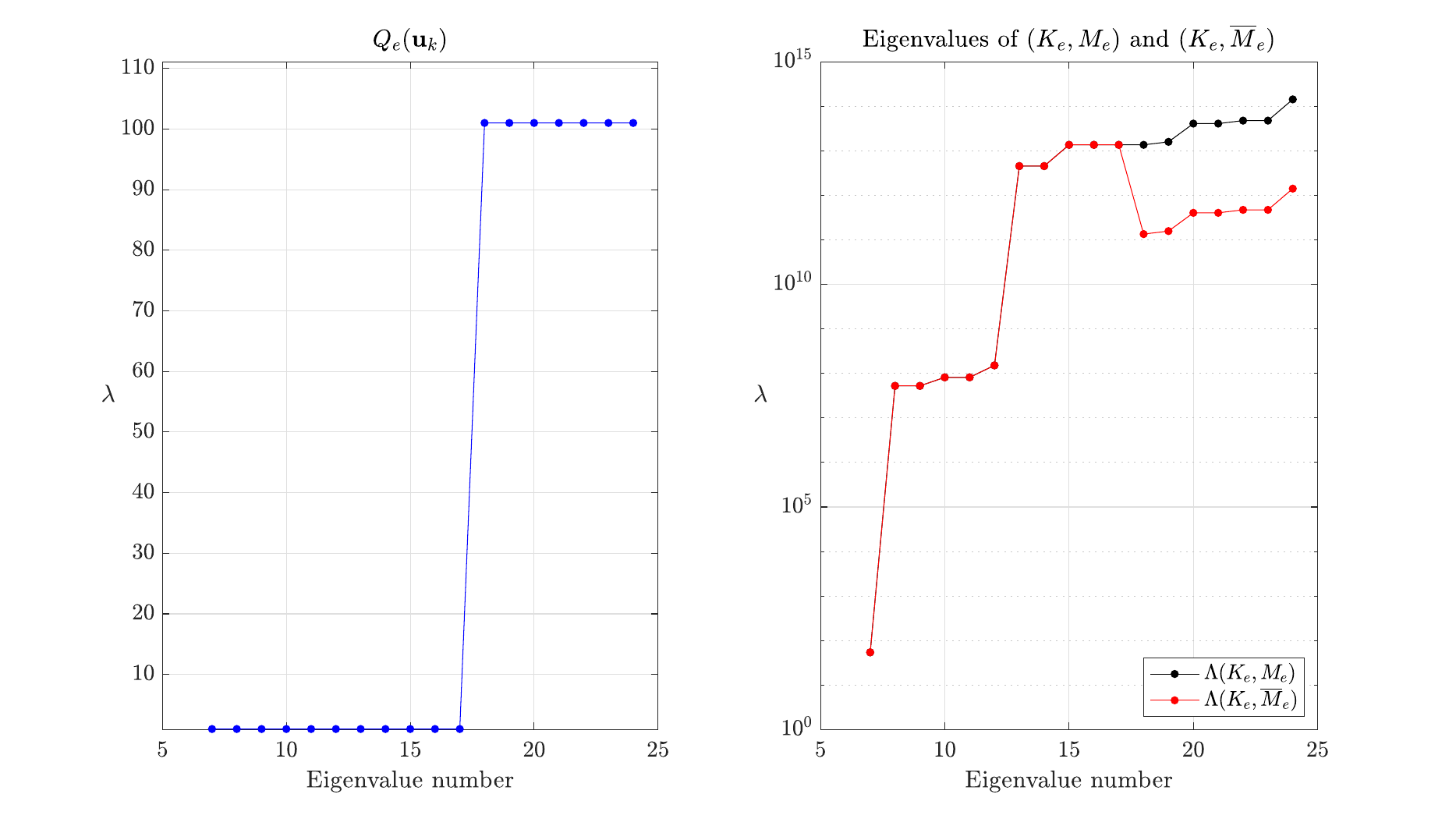}
    \caption{Values of $Q_e(\mathbf{u}_k)$ for $k=7,\dots,m$ and eigenvalues of $(K_e,M_e)$ and $(K_e,\overline{M}_e)$ for the first strategy with $r=7$ and $\alpha=10^2$}
    \label{fig: 3D_Elasticity_plate_FEM_local_defl_st1_eig_element_r7_alpha1e2}
\end{figure}

Choosing a large value of $\alpha$ may quickly undermine the quality of the solution as it moves inaccurate high frequency modes to the low frequency regime. González and Park \cite{gonzalez2020large} already highlighted the limitations of this technique for a simple bar model and its alarming impact on the accuracy. Thus, one should exercise caution when using this technique. The second strategy, on the contrary, always preserves the eigenvalue ordering. The result, illustrated in \Cref{fig: 3D_Elasticity_plate_FEM_local_defl_st2_eig_element_r12} for $r=12$, consists in shaving off the upper part of the spectrum.

\begin{figure}[H]
    \centering
    \includegraphics[width=0.8\linewidth]{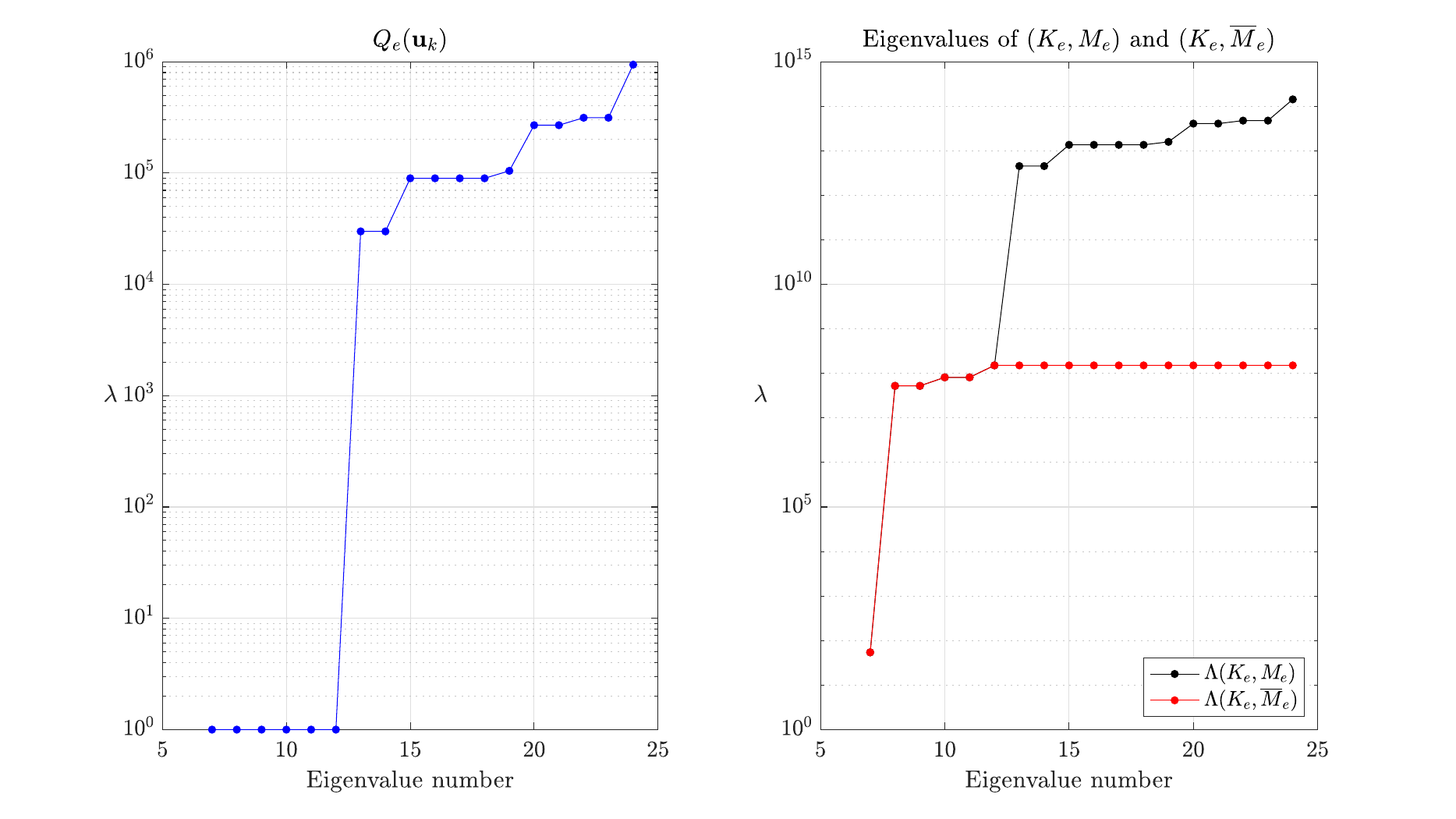}
    \caption{Values of $Q_e(\mathbf{u}_k)$ for $k=7,\dots,m$ and eigenvalues of $(K_e,M_e)$ and $(K_e,\overline{M}_e)$ for the second strategy with $r=12$}
    \label{fig: 3D_Elasticity_plate_FEM_local_defl_st2_eig_element_r12}
\end{figure}

We now investigate the global properties of the method using the same finite element mesh and material properties as in the previous section. \Cref{fig: 3D_Elasticity_plate_FEM_local_defl_st1_delta_t_crit_hex8} shows the increase of the critical time step for the first strategy for fixed rank values and increasing values of $\alpha$ together with the upper bound from \Cref{cor: local_deflation_st1}. Interestingly, \Cref{fig: 3D_Elasticity_plate_FEM_local_defl_st1_delta_t_crit_hex8} displays a threshold effect reminiscent of the local behavior. \Cref{fig: 3D_Elasticity_plate_FEM_local_defl_st2_delta_t_crit_hex8} shows the increase of the critical time step for the second strategy, which is perfectly captured by the upper bound of \Cref{cor: local_deflation_st2}.

\begin{figure}[H]
    \centering
    \begin{subfigure}[t]{0.48\textwidth}
    \centering
    \includegraphics[width=\textwidth]{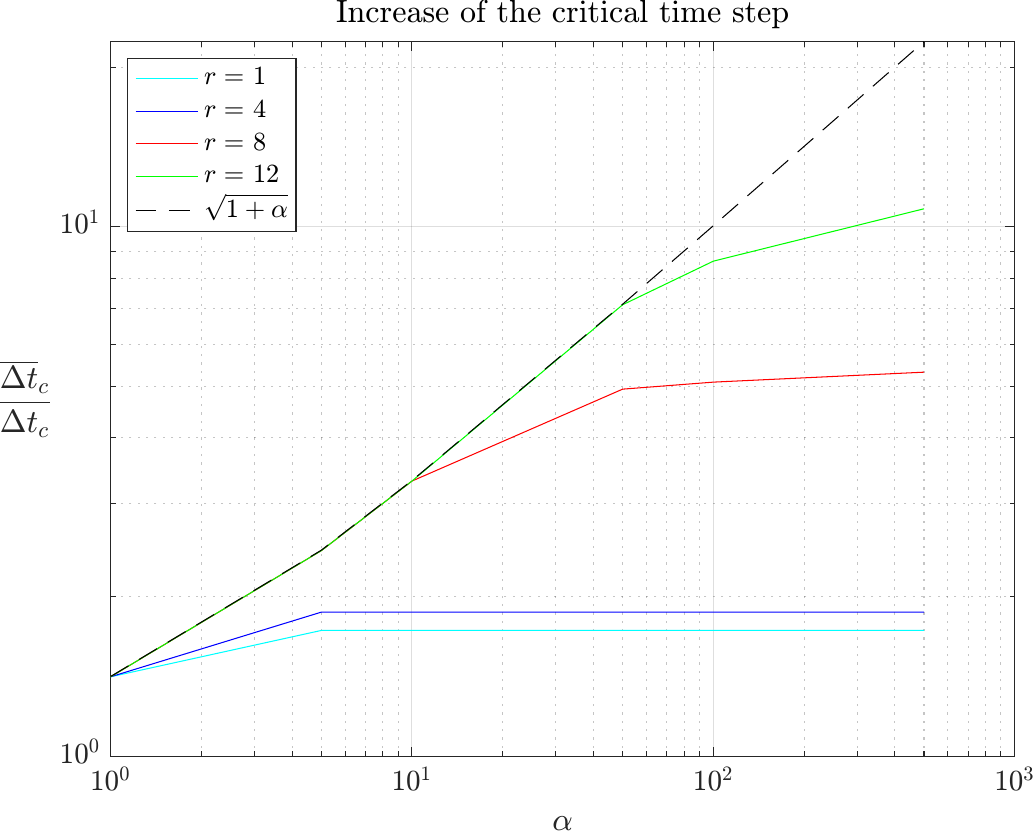}
    \caption{Strategy 1}
    \label{fig: 3D_Elasticity_plate_FEM_local_defl_st1_delta_t_crit_hex8}
    \end{subfigure}
    \hfill
    \begin{subfigure}[t]{0.48\textwidth}
    \centering
    \includegraphics[width=\textwidth]{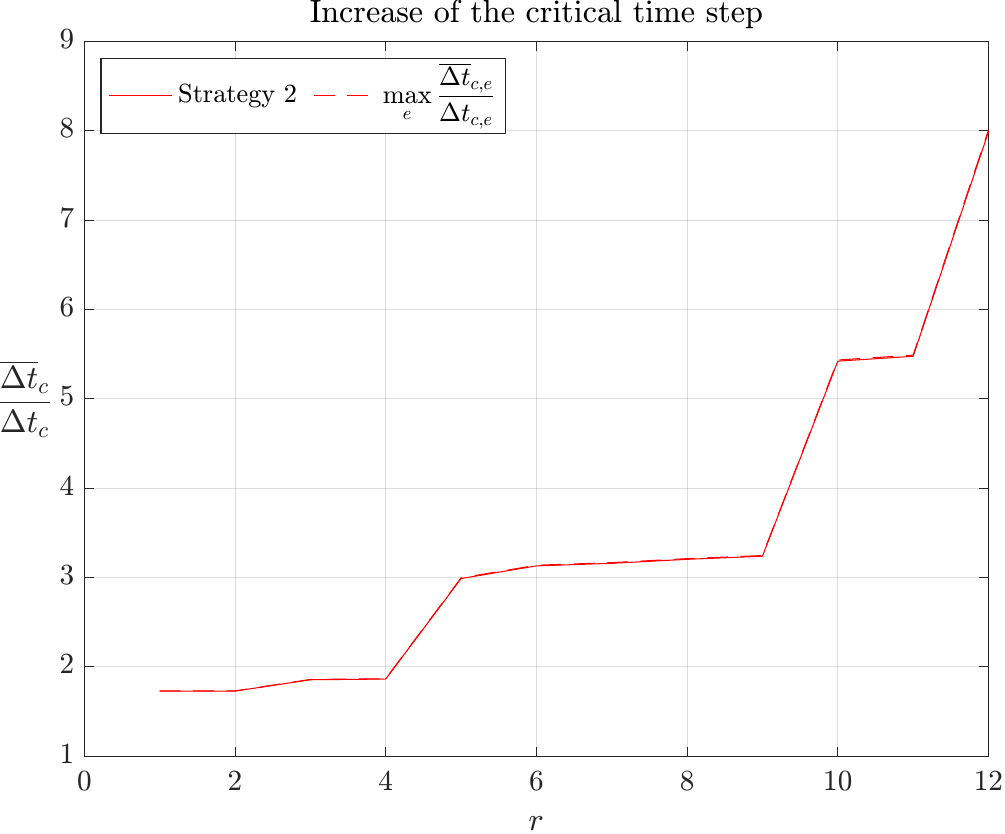}
    \caption{Strategy 2}
    \label{fig: 3D_Elasticity_plate_FEM_local_defl_st2_delta_t_crit_hex8}
    \end{subfigure}
    \hfill
    \caption{Increase of the critical time step for local deflation strategies}
    \label{fig: 3D_Elasticity_plate_FEM_local_defl_delta_t_crit_hex8}
\end{figure}

The eigenfrequency ratio for the entire spectrum is shown in \Cref{fig: 3D_Elasticity_plate_FEM_local_defl_eig_ratio_hex8} for both strategies and selected values of scaling parameter $\alpha$ and rank $r$. Contrary to the method of Olovsson et al., the lower bound is sharp only for the largest eigenfrequencies, which is the desired behavior.

\begin{figure}[H]
    \centering
    \begin{subfigure}[t]{0.48\textwidth}
    \centering
    \includegraphics[width=\textwidth]{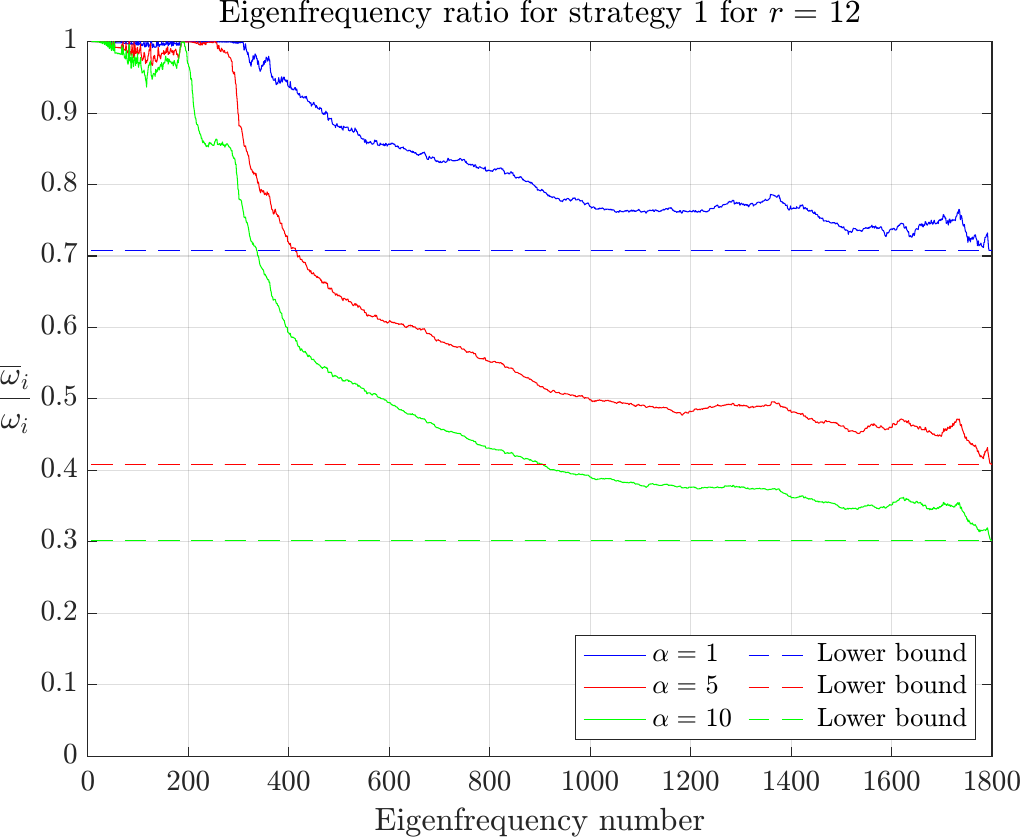}
    \caption{Strategy 1}
    \label{fig: 3D_Elasticity_plate_FEM_local_defl_st1_eig_ratio_hex8_r12}
    \end{subfigure}
    \hfill
    \begin{subfigure}[t]{0.48\textwidth}
    \centering
    \includegraphics[width=\textwidth]{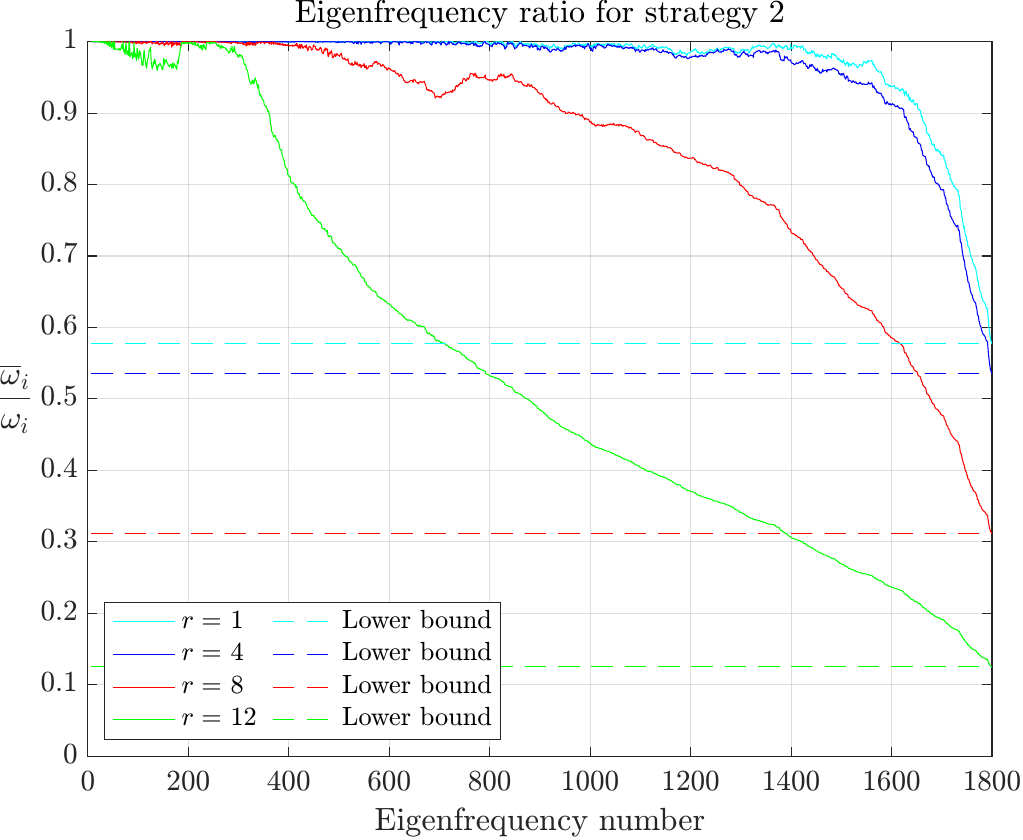}
    \caption{Strategy 2}
    \label{fig: 3D_Elasticity_plate_FEM_local_defl_st2_eig_ratio_hex8}
    \end{subfigure}
    \hfill
    \caption{Eigenfrequency ratio}
    \label{fig: 3D_Elasticity_plate_FEM_local_defl_eig_ratio_hex8}
\end{figure}

Admittedly, the first strategy cannot significant increase the critical time step without undermining the accuracy of the lower frequencies. However, the second strategy for deflation ranks $r=1,\dots,8$ seems quite promising as the eigenfrequency ratio abruptly decays only towards the end of the spectrum, ensuring both good accuracy and a significant increase of the critical time step.

Finally, we note that the proof strategy of \Cref{cor: bounds_olovsson,cor: bounds_hoffmann} also immediately yields (not necessarily tight) upper bounds on the condition number of the scaled mass matrix for these methods.

\subsection{Summary}
The ad hoc mass scaling strategies of Olovsson et al. and Hoffmann et al. may be viewed as inexact local deflation strategies based on low-frequency modes. One notable advantage of using low-frequency modes instead of high-frequency ones is that the former includes rigid-body modes and are known regardless of how distorted the element is. Thus, they offer better guarantees of robustness. Moreover, engineers have developed a very good understanding of the deformation mode shapes for an hexahedral element (see \citep[][Figure 3]{cocchetti2013selective}), based on which they constructed these ad hoc strategies. While the method of Hoffmann et al. seems significantly better than the one of Olovsson et al., it still does not exactly target the right modes, which both impedes on the accuracy and the increase of the critical time step. Preserving the eigenvalue ordering seems essential to avoid interchanging modes, which might disastrously affect the accuracy of the solution. Exact local deflation strategies may deliver far better results but assume that the stiffness matrix is explicitly available, which is rarely the case in commercial software, especially for nonlinear problems. We believe it might be possible to further improve the method of Hoffmann et al. for obtaining the desired behavior but leave this task to the engineering community.

The experiments have also revealed that the global properties are tightly connected to the local ones. The global matrix pairs $(K,M)$ and $(M,\overline{M})$ may nearly share some eigenspaces but we have not explored this hypothesis any further.

\section{Conclusion}
\label{se: conclusion}
Selective mass scaling techniques are increasingly popular in structural dynamics for increasing the critical time step of explicit time integration methods, while preserving the accuracy of structurally important lower frequencies and mode shapes. While physical intuition is sometimes sufficient for identifying and eliminating the constraint on the step size, we have substantiated it with a strong theoretical analysis thereby unraveling appealing properties of some of the most popular methods. In particular, we have drawn a unifying picture and provided a clearer positioning of the methods by bridging the gap between different communities and highlighting fundamental similarities and differences. Practically speaking though, while some methods are certainly valuable to the engineering community, others are mostly impractical and often based on elementary linear algebra facts known since decades. Moreover, nearly all methods lead to a non-diagonal mass matrix and defeat the spirit of explicit dynamics. The methods are only competitive if the saving in the number of iterations overtakes the additional cost per iteration for repeatedly solving linear systems with the scaled mass matrix. Directly modifying the inverse mass matrix \citep{olovsson2004selective,cocchetti2013selective,cocchetti2015selective,tkachuk2015direct,gonzalez2018inverse} or the stiffness matrix \cite{voet2024robust} suggests itself as a promising pathway. Finally, the benefits and drawbacks of mass scaling must be carefully weighted on a case by case basis. While mass scaling may certainly help in some cases, it might also irremediably degrade the solution, depreciating advanced material models and finite elements if not undermining the very purpose of numerical simulations. We feel this issue is largely overlooked, especially for nonlinear problems, and advise caution when using such methods.

\section*{Declaration of competing interest}
The authors declare that they have no known competing financial interests or personal relationships that could have appeared to influence the work reported in this paper.

\section*{Acknowledgments}
The second author kindly acknowledges the support of the SNSF through the project ``Smoothness in Higher Order Numerical Methods’’ n. TMPFP2\_209868 (SNSF Swiss Postdoctoral Fellowship 2021). 
\newline The third author kindly acknowledges the support of SNSF through the project ``PDE tools for analysis-aware geometry processing in simulation science’’ n. 200021\_215099.

\end{document}